\documentclass[a4paper, reqno, 10pt]{amsart}

\usepackage{verbatim}
\usepackage{kotex}
\usepackage[dvips]{color}
\usepackage[usenames,dvipsnames]{xcolor}
\usepackage{amssymb}
\usepackage[active]{srcltx}
\usepackage[colorlinks,pdfpagelabels,pdfstartview = FitH,bookmarksopen = true,bookmarksnumbered = true,linkcolor = NavyBlue,plainpages = false,hypertexnames = false,citecolor = OliveGreen] {hyperref}
\usepackage[italian, textsize=tiny, color=BlueGreen, backgroundcolor=BlueGreen, linecolor=BlueGreen]{todonotes}
\usepackage{mathtools}
\usepackage{empheq}
\usepackage{a4wide}
\usepackage{IEEEtrantools}

\newcommand\blfootnote[1]{%
\begingroup
\renewcommand\thefootnote{}\footnote{#1}%
\addtocounter{footnote}{-1}%
\endgroup
}

\allowdisplaybreaks

\title[Nonlocal double phase]{H\"older regularity for weak solutions to nonlocal double phase problems}

\author[Byun]{Sun-Sig Byun}
\address{Department of Mathematical Sciences and Research Institute of Mathematics,
Seoul National University, Seoul 08826, Korea}
\email{byun@snu.ac.kr}

\author[Ok]{Jihoon Ok}
\address{Department of Mathematics, Sogang University, Seoul 04107, Korea}
\email{jihoonok@sogang.ac.kr}

\author[Song]{Kyeong Song}
\address{Department of Mathematical Sciences,
Seoul National University, Seoul 08826, Korea}
\email{kyeongsong@snu.ac.kr}



\newtheorem{theorem}{Theorem}[section]

\newtheorem{lemma}[theorem]{Lemma}

\newtheorem{corollary}[theorem]{Corollary}
\theoremstyle{definition}

\newtheorem{remark}[theorem]{Remark}

\numberwithin{equation}{section}

\def\eqn#1$$#2$${\begin{equation}\label#1#2\end{equation}}
\def\charfn_#1{{\raise1.2pt\hbox{$\chi_{\kern-1pt\lower3pt\hbox{{$\scriptstyle#1$}}}$}}}

\makeatletter
\newcommand{\pushright}[1]{\ifmeasuring@#1\else\omit\hfill$\displaystyle#1$\fi\ignorespaces}
\newcommand{\pushleft}[1]{\ifmeasuring@#1\else\omit$\displaystyle#1$\hfill\fi\ignorespaces}
\makeatother

\newcommand{\R}{\mathbb{R}}

 \DeclareMathOperator*{\osc}{osc}

\def\diam{\operatorname{diam}}

\newcommand{\data}{\texttt{data}}

\def\loc{{\operatorname{loc}}}

\newcommand{\supp}{{\rm supp}}

\def\mean#1{\mathchoice%
          {\mathop{\kern 0.2em\vrule width 0.6em height 0.69678ex depth -0.58065ex
                  \kern -0.8em \intop}\nolimits_{\kern -0.4em#1}}%
          {\mathop{\kern 0.1em\vrule width 0.5em height 0.69678ex depth -0.60387ex
                  \kern -0.6em \intop}\nolimits_{#1}}%
          {\mathop{\kern 0.1em\vrule width 0.5em height 0.69678ex
              depth -0.60387ex
                  \kern -0.6em \intop}\nolimits_{#1}}%
          {\mathop{\kern 0.1em\vrule width 0.5em height 0.69678ex depth -0.60387ex
                  \kern -0.6em \intop}\nolimits_{#1}}}

\def\vintslides_#1{\mathchoice%
          {\mathop{\kern 0.1em\vrule width 0.5em height 0.697ex depth -0.581ex
                  \kern -0.6em \intop}\nolimits_{\kern -0.4em#1}}%
          {\mathop{\kern 0.1em\vrule width 0.3em height 0.697ex depth -0.604ex
                  \kern -0.4em \intop}\nolimits_{#1}}%
          {\mathop{\kern 0.1em\vrule width 0.3em height 0.697ex depth -0.604ex
                  \kern -0.4em \intop}\nolimits_{#1}}%
          {\mathop{\kern 0.1em\vrule width 0.3em height 0.697ex depth -0.604ex
                  \kern -0.4em \intop}\nolimits_{#1}}}

\newcommand{\aveint}[2]{\mathchoice%
          {\mathop{\kern 0.2em\vrule width 0.6em height 0.69678ex depth -0.58065ex
                  \kern -0.8em \intop}\nolimits_{\kern -0.45em#1}^{#2}}%
          {\mathop{\kern 0.1em\vrule width 0.5em height 0.69678ex depth -0.60387ex
                  \kern -0.6em \intop}\nolimits_{#1}^{#2}}%
          {\mathop{\kern 0.1em\vrule width 0.5em height 0.69678ex depth -0.60387ex
                  \kern -0.6em \intop}\nolimits_{#1}^{#2}}%
          {\mathop{\kern 0.1em\vrule width 0.5em height 0.69678ex depth -0.60387ex
                  \kern -0.6em \intop}\nolimits_{#1}^{#2}}}

\newtoks\by
\newtoks\paper
\newtoks\book
\newtoks\jour
\newtoks\yr
\newtoks\pages
\newtoks\vol
\newtoks\publ

\def\ota{{\hbox{\bf ???}}}
\def\cLear{\by=\ota\paper=\ota\book=\ota\jour=\ota\yr=\ota
\pages=\ota\vol=\ota\publ=\ota}
\def\endpaper{\the\by, \textit{\the\paper},
{\the\jour} \textbf{\the\vol} (\the\yr), \the\pages.\cLear}
\def\endbook{\the\by, \textit{\the\book},
\the\publ, \the\yr.\cLear}
\def\endpap{\the\by, \textit{\the\paper}, \the\jour.\cLear}
\def\endproc{\the\by, \textit{\the\paper}, \the\book, \the\publ,
\the\yr, \the\pages.\cLear}

\begin{document}
\maketitle
\begin{abstract}
We prove local boundedness and H\"older continuity for weak solutions to nonlocal double phase problems concerning the following fractional energy functional
$$
\int_{\R^n}\int_{\R^n} \frac{|v(x)-v(y)|^p}{|x-y|^{n+sp}} + a(x,y)\frac{|v(x)-v(y)|^q}{|x-y|^{n+tq}}\, dxdy,
$$
where $0<s\le t<1<p \leq q<\infty$ and $a(\cdot,\cdot) \geq 0$. For such regularity results, we identify sharp assumptions on the modulating coefficient $a(\cdot,\cdot)$ and the powers $s,t,p,q$ which are analogous to those for local double phase problems. 
\end{abstract}

\blfootnote{S.-S. Byun was supported by NRF-2021R1A4A1027378. J. Ok was supported by NRF-2017R1C1B2010328. K. Song was supported by NRF-2020R1C1C1A01014904. }

\section{Introduction} \label{introduction}
In this paper, we study the regularity theory for weak solutions to the following integro-differential equation:
\begin{equation}\label{main.eq}
\begin{aligned}
\mathcal{L}u(x) &\coloneqq \mathrm{P.V.}\int_{\mathbb{R}^{n}} |u(x)-u(y)|^{p-2}(u(x)-u(y))K_{sp}(x,y)\,dy \\
& \qquad + \mathrm{P.V.}\int_{\mathbb{R}^{n}}a(x,y)|u(x)-u(y)|^{q-2}(u(x)-u(y))K_{tq}(x,y)\,dy = 0 \quad\textrm{in }\, \Omega. 
\end{aligned}
\end{equation}
Here, $\Omega \subset \mathbb{R}^{n}$ ($n\geq 2$) is a bounded domain, $K_{sp},K_{tq}: \mathbb{R}^{n}\times\mathbb{R}^{n}\rightarrow \mathbb{R}$ are suitable kernels with orders $(s,p)$ and $(t,q)$, respectively, for some $0<s \le t <1< p \leq q<\infty$, and $a:\mathbb{R}^{n}\times\mathbb{R}^{n}\rightarrow \mathbb{R}$ is a nonnegative modulating coefficient. We note that when $a(x,y) \equiv 0$ and $K_{sp}(x,y) \equiv |x-y|^{-(n+sp)}$, it reduces to the ($s$-)fractional $p$-Laplace equation, $(-\Delta)^{s}_{p}u = 0$. Precise assumptions will be described in Section \ref{setting} below.

The regularity theory for nonlocal problems with fractional orders has been extensively studied for the last two decades. Caffarelli and Silvestre  \cite{CS1} proved Harnack inequality for the fractional Laplace equation, $(-\Delta)^{s}u \coloneqq (-\Delta)^{s}_{2}u=0$, by using an extension argument. Later, Caffarelli, Chan and Vasseur \cite{CCV1} applied De Giorgi's approach to 
linear parabolic equations with fractional orders 
involving general kernels $K_{sp}$ with $p=2$, and proved H\"older continuity of weak solutions. We refer to, for instance, \cite{CS2,Ka1,Ka2,KMS,PP1,PP2,PPS,PS, Silv} for regularity results for nonlocal linear equations with fractional orders. For nonlocal equations of fractional $p$-Laplacian type, Di Castro, Kuusi and Palatucci \cite{DKP1,DKP2} employed a nonlocal version of De Giorgi's approach to prove local H\"older continuity and Harnack inequality. Cozzi \cite{coz} extended  these results to non-homogeneous problems with lower order terms, by using fractional De Giorgi classes. Such approaches are further applied to several research areas including obstacle problems \cite{MR3503212} and measure data problems \cite{MR3339179}. We also refer to \cite{DZZ1,GaKi1,KKP1,KKP2,Lind,No1,No2} and references therein for various regularity results for nonlocal problems of fractional $p$-Laplacian type. For a general overview of the history and related topics, we refer to the survey paper \cite{palatucci2018} and the monographs \cite{BV, KP}.

Nonlocal problems with nonstandard growth are getting more and more attention in the very recent years. For local problems, there are two typical models of nonstandard growth conditions. One is the variable growth condition concerned with the function $t^{p(x)}$; the other is the double phase growth condition concerned with the function $t^{p}+a(x)t^{q}$.  Recently, the second author in this paper \cite{ok2021local} proved the local H\"older continuity  of weak solutions to nonlocal equations with variable growth by developing the technique used in \cite{DKP2}. We would like to mention a recent paper \cite{chaker2021local} in which a similar result was obtained with more restrictive assumptions. We also refer to the above papers and references therein for the research on nonlocal problems with variable growth or relevant function spaces.

A prototype of nonlocal double phase problems is the following equation:
\begin{equation}\label{model}
\begin{aligned}
\mathrm{P.V.}&\int_{\mathbb{R}^{n}} \frac{|u(x)-u(y)|^{p-2}(u(x)-u(y))}{|x-y|^{n+sp}}\,dy \\ 
&+\mathrm{P.V.}\int_{\mathbb{R}^{n}} a(x,y)\frac{|u(x)-u(y)|^{q-2}(u(x)-u(y))}{|x-y|^{n+tq}}\,dy = 0 \quad \textrm{in } \, \Omega,
\end{aligned}
\end{equation}
which is the case when $K_{sp}(x,y) \equiv |x-y|^{-(n+sp)}$ and $K_{tq}(x,y) \equiv |x-y|^{-(n+tq)}$ in \eqref{main.eq}. It is in fact the Euler-Lagrange equation of the functional
\begin{equation}\label{model.functional}
v \mapsto \iint_{\mathcal{C}_{\Omega}}\frac{1}{p}\frac{|v(x)-v(y)|^{p}}{|x-y|^{n+sp}} + a(x,y)\frac{1}{q}\frac{|v(x)-v(y)|^{q}}{|x-y|^{n+tq}}\,dxdy,
\end{equation}
where
\begin{equation}\label{COmega}
\mathcal{C}_{\Omega} \coloneqq (\mathbb{R}^{n}\times\mathbb{R}^{n})\setminus((\mathbb{R}^{n}\setminus\Omega)\times(\mathbb{R}^{n}\setminus\Omega)).
\end{equation}

The local version corresponding to \eqref{model} is the double phase equation
\begin{equation}\label{localdp}
\mathrm{div}\left(|Du|^{p-2}Du+a(x)|Du|^{q-2}Du\right) = 0 \quad \textrm{in } \, \Omega.
\end{equation}
Starting from \cite{MR3360738,MR3294408}, the regularity for weak solutions to \eqref{localdp} and minimizers of corresponding variational integral
has been exhaustively studied, see \cite{BCM2,BO, MR3447716, MR3985927, DO, ok2} and references therein. In particular, for local boundedness and  H\"older continuity, it is shown that, when $1<p\le n$,
\begin{equation}\label{localdp.condition}
\begin{aligned}
a(\cdot) \in L^{\infty}_{\loc}(\Omega), \,\, q \leq p^{*}  \;\; &\Longrightarrow \;\; u \in L^{\infty}_{\mathrm{loc}}(\Omega),\\
u \in L^{\infty}_{\mathrm{loc}}(\Omega),\,\, a(\cdot) \in C^{0,\alpha}_{\loc}(\Omega), \,\, q \leq p + \alpha \;\; &\Longrightarrow \;\; u \in C^{0,\gamma}_{\mathrm{loc}}(\Omega),
\end{aligned}
\end{equation}
see \cite{BCM1,MR3360738,MR3366102}.

Nonlocal equations of double phase type were first treated by De Filippis and Palatucci \cite{MR3944281}, where they proved H\"older continuity for viscosity solutions. Scott and Mengesha \cite{scott2020self} proved nonlocal self-improving property for bounded weak solutions. We also mention the paper \cite{fang2021weak} by Fang and Zhang  concerning H\"older continuity for bounded weak solutions and a relationship between weak and viscosity solutions. Here, we point out that the above papers \cite{MR3944281,fang2021weak,scott2020self} are restricted to solutions bounded in $\mathbb{R}^{n}$ and are under the assumption that $t\le s$
, which means the second term in \eqref{model} is a lower order term. Therefore, they were able to consider bounded, possibly discontinuous modulating coefficient $a(\cdot,\cdot)$. 

 In this paper we prove the local boundedness and H\"older continuity for weak solutions to the nonlocal equation with double phase growth condition, \eqref{main.eq}. We emphasize that we deal with the case $s\le t$, 
which is a main difference from the papers \cite{MR3944281,fang2021weak,scott2020self}. The case $s\le t$ is more delicate than the other case, since the second term in \eqref{model.functional} has a higher order in the sense that $t \ge s$, $q \ge p$. 
To the best of our knowledge, the results presented in this paper are the first regularity results in this case. 
When we prove H\"older continuity in this case, we assume that the modulating coefficient $a(\cdot,\cdot)$ is  H\"older continuous, which together with a restriction of the range of $q$ allows us to replace $a(\cdot,\cdot)$ with a constant. Note that this  argument is exactly the same as the one for the local double phase problem. Therefore, we are able to make the assumptions of $q$ and $a(\cdot,\cdot)$ that are analogous to those in \eqref{localdp.condition}.
 Moreover, we only assume that the weak solution is not bounded in $\mathbb{R}^{n}$, but locally bounded in $\Omega$. Therefore, we need to handle the so-called nonlocal tails. The main difficulty arises in deriving the logarithmic type estimate (see Lemma~\ref{log.lemma}). For  fractional $p$-Laplacian type problems, an analogue estimate was obtained in \cite[Lemma 1.3]{DKP2}. However, we could not apply the same approach directly to our problem \eqref{main.eq}. In order to obtain such an estimate, we first assume that the weak solution is locally bounded, and then take advantage of the H\"older continuity of $a(\cdot,\cdot)$ in order to modify and develop the techniques used in the proof of \cite[Lemma 1.3]{DKP2}.

\subsection{Assumptions and main results}\label{setting}

We say that a 
function $f:\mathbb{R}^{n}\times\mathbb{R}^{n}\rightarrow \mathbb{R}$ is 
symmetric if $f(x,y)=f(y,x)$ for every $x,y \in \mathbb{R}^{n}$.

The kernels $K_{sp},K_{tq}:\mathbb{R}^{n}\times\mathbb{R}^{n}\rightarrow \mathbb{R}$ are measurable, symmetric 
and satisfy 
\begin{align}\label{kernel.growth}
\frac{\Lambda^{-1}}{|x-y|^{n+sp}} \leq K_{sp}(x,y) \leq \frac{\Lambda}{|x-y|^{n+sp}}, \qquad
\frac{\Lambda^{-1}}{|x-y|^{n+tq}} \leq K_{tq}(x,y) \leq \frac{\Lambda}{|x-y|^{n+tq}}
\end{align}
for a.e. $(x,y)\in\mathbb{R}^{n}\times\mathbb{R}^{n}$, where $\Lambda>1$ and
\begin{equation}\label{power}
1 < p \leq q<\infty, \qquad 0 < s \leq t < 1.
\end{equation} 
The modulating coefficient $a:\mathbb{R}^{n}\times\mathbb{R}^{n} \rightarrow \mathbb{R}$ is assumed to be nonnegative, measurable, symmetric and bounded:
\begin{equation}\label{a.bound}
0 \leq a(x,y) = a(y,x) \leq \lVert a \rVert_{L^{\infty}}, \qquad x,y \in \mathbb{R}^{n}.
\end{equation} 
In addition, in Theorem~\ref{thm.hol} and Section~\ref{Holder},  we  also assume that 
\begin{equation}\label{a.holder}
|a(x_{1},y_{1}) - a(x_{2},y_{2})| \leq [a]_{\alpha}(|x_{1}-x_{2}|+|y_{1}-y_{2}|)^{\alpha}, \quad  \alpha>0,
\end{equation}
for every $(x_{1},y_{1}),(x_{2},y_{2})\in\mathbb{R}^{n}\times\mathbb{R}^{n}$. 

With the relevant function spaces including $\mathcal{A}(\Omega)$ and $L^{q-1}_{sp}(\mathbb{R}^{n})$ to be introduced in the next section, we introduce weak solutions under consideration.  
We say that $u \in \mathcal{A}(\Omega)$ is a weak solution to \eqref{main.eq} if
\begin{equation}\label{weak.formulation}
\begin{aligned}
\iint_{\mathcal{C}_{\Omega}}&\left[|u(x)-u(y)|^{p-2}(u(x)-u(y))(\varphi(x)-\varphi(y))K_{sp}(x,y)\right. \\
&\left.\quad+a(x,y)|u(x)-u(y)|^{q-2}(u(x)-u(y))(\varphi(x)-\varphi(y))K_{tq}(x,y)\right]\,dxdy = 0
\end{aligned}
\end{equation}
for every $\varphi \in \mathcal{A}(\Omega)$ with $\varphi = 0$ a.e. in $\mathbb{R}^{n}\setminus\Omega$.  In addition, we say that $u \in \mathcal{A}(\Omega)$ is a weak subsolution (resp. supersolution) if \eqref{weak.formulation} with ``='' replaced by ``$\leq$ (resp. $\geq$)'' holds for every $\varphi \in \mathcal{A}(\Omega)$ satisfying $\varphi \geq 0$ a.e. in $\mathbb{R}^{n}$ and $\varphi = 0$ a.e. in $\mathbb{R}^{n}\setminus\Omega$. Existence and uniqueness of weak solutions to \eqref{main.eq} with a Dirichlet  boundary condition  will be  discussed in Section \ref{existence}.

Now we state our main results. The first one is the local boundedness of weak solutions. 
\begin{theorem}\label{thm.bdd}
Let $K_{sp},K_{tq},a:\mathbb{R}^{n}\times\mathbb{R}^{n}\rightarrow\mathbb{R}$ be symmetric and satisfy \eqref{kernel.growth}-\eqref{a.bound}.
If 
\begin{equation}\label{assumption.bdd}
\begin{cases}
 p\le q \leq  \frac{np}{n-sp} & \text{when }\ sp<n,\\  
p\le q< \infty & \text{when }\ sp\ge n,
\end{cases}
\end{equation} 
then every weak solution $u \in \mathcal{A}(\Omega)\cap L^{q-1}_{sp}(\mathbb{R}^{n})$ to \eqref{main.eq} is locally bounded in $\Omega$.
\end{theorem}


The second one is the local H\"older continuity. Here, we assume that $a(\cdot,\cdot)$ is H\"older continuous in $\R^n\times \R^n$ and that $u$ is locally bounded in $\Omega$.

\begin{theorem}\label{thm.hol}
Let $K_{sp},K_{tq},a:\mathbb{R}^{n}\times\mathbb{R}^{n}\rightarrow\mathbb{R}$ be symmetric and satisfy \eqref{kernel.growth}-\eqref{a.bound}. 
If $a(\cdot,\cdot)$ satisfies \eqref{a.holder} and 
\begin{equation}\label{assumption.hol}
\quad tq \leq sp + \alpha,
\end{equation}
then every weak solution $u \in \mathcal{A}(\Omega)\cap L^{q-1}_{sp}(\mathbb{R}^{n})$ to \eqref{main.eq} which is locally bounded in $\Omega$ is locally H\"older continuous in $\Omega$. More precisely, for every open subset $\Omega'\Subset\Omega$, there exists $\gamma\in (0,1)$ depending only on $n,s,t,p,q,\Lambda,\lVert a \Vert_{L^{\infty}},\alpha,[a]_{\alpha}$ and $\lVert u \rVert_{L^{\infty}(\Omega')}$ such that $u \in C^{0,\gamma}_{\mathrm{loc}}(\Omega')$.
\end{theorem}

\begin{remark}\label{rmk0}
In view of Theorem~\ref{thm.bdd}, we also see that, under the setting in Theorem~\ref{thm.hol}, if
\[
\begin{cases}
 p\le q \leq \min\left\{ \frac{np}{n-sp}, \frac{sp+\alpha}{t} \right\}& \text{when }\ sp<n,\\  
p\le q\le  \frac{sp+\alpha}{t} =  \frac{n+\alpha}{t} & \text{when }\ sp = n,
\end{cases}
\]
then every weak solution $u \in \mathcal{A}(\Omega)\cap L^{q-1}_{sp}(\mathbb{R}^{n})$ to \eqref{main.eq} is locally H\"older continuous.
\end{remark}



This paper is organized as follows: In the next section, we introduce basic notation and function spaces which will be used throughout this paper. In Section \ref{existence}, we examine the existence of solutions to \eqref{main.eq}. In Section \ref{Caccio.bounded} we derive a Caccioppoli type estimate and prove Theorem \ref{thm.bdd}. Finally, in Section \ref{Holder}, we prove  Theorem \ref{thm.hol} by obtaining a logarithmic estimate.
\section{Preliminaries}\label{preliminaries}
\subsection{Notation}
We denote by $c$ a generic constant greater than or equal to one, whose value may vary from line to line. We denote its specific dependence in parentheses when needed, using the abbreviation like
\begin{equation*}
\begin{cases}
\data \coloneqq (n,s,t,p,q,\Lambda,\lVert a \rVert_{L^{\infty}}) \\
\data_{1} \coloneqq (n,s,t,p,q,\Lambda, \lVert a \rVert_{L^{\infty}}, \alpha, [a]_{\alpha}),
\end{cases}
\end{equation*}
where $\|a\|_{L^{\infty}}$, $\alpha$ and $[a]_{\alpha}$ are given in \eqref{a.bound} and \eqref{a.holder}.

For any open set $\mathcal{O}\subseteq \mathbb{R}^{n}$, $s\in(0,1)$ and $p\geq 1$, the fractional Sobolev space $W^{s,p}(\mathcal{O})$ is the set of all functions $v \in L^{p}(\mathcal{O})$ for which 
\begin{align*}
\lVert v \rVert_{W^{s,p}(\mathcal{O})} & \coloneqq \lVert v \rVert_{L^{p}(\mathcal{O})} + [v]_{s,p;\mathcal{O}} \\
& \coloneqq \left(\int_{\mathcal{O}}|v|^{p}\,dx\right)^{\frac{1}{p}} + \left(\int_{\mathcal{O}}\int_{\mathcal{O}}\frac{|v(x)-v(y)|^{p}}{|x-y|^{n+sp}}\,dxdy\right)^{\frac{1}{p}} < \infty.
\end{align*}
Furthermore, we define $W^{s,p}_{0}(\mathcal{O})$ as the closure of $C^{\infty}_{0}(\mathcal{O})$ in $W^{s,p}(\mathcal{O})$.
We denote the ($s$-)fractional Sobolev conjugate of $p$ by 
\begin{equation*}
p_{s}^{*} \coloneqq 
\begin{cases}
np/(n-sp) & \text{when }\ sp < n, \\
\textrm{any number in }(p,\infty) & \text{when }\ sp \ge n.
\end{cases}
\end{equation*} 
In particular, if we consider two exponents $p$ and $q$ with $1<p<q$, we set $p^*_s=q+1$ when $sp=n$.

As usual, $B_{r}(x_{0})$ is the open ball in $\mathbb{R}^{n}$ with center $x_{0} \in \mathbb{R}^{n}$ and radius $r>0$. We omit the center when it is clear in the context. 
For a 
measurable function $v$, we write $v_{\pm} \coloneqq \max\{\pm v,0\}$. If $v$ is integrable over a measurable set $S$ with $0<|S|<\infty$, we denote its integral average over $S$ by
\begin{equation*}
(v)_{S} \coloneqq \mean{S}v\,dx \coloneqq \frac{1}{|S|}\int_{S}v\,dx.
\end{equation*}

We always assume that $s$, $t$, $p$, and $q$ satisfy \eqref{power} and that $K_{sp},K_{tq},a:\R^n\times \R^n \rightarrow \mathbb{R}$ satisfy \eqref{kernel.growth} and \eqref{a.bound}. 
We denote
\begin{equation}\label{def.H}
H(x,y,\tau)\coloneqq\frac{\tau^{p}}{|x-y|^{sp}}+ a(x,y) \frac{\tau^{q}}{|x-y|^{tq}}, 
\quad x,y\in\R^n \ \text{ and }\ \tau\geq 0 ,
\end{equation}  
and
\begin{equation}\label{modular}
\varrho(v;S) \coloneqq \int_{S}\int_{S}H(x,y,|v(x)-v(y)|)\,\frac{dxdy}{|x-y|^{n}}
\end{equation}
for each measurable set $S\subseteq \mathbb{R}^{n}$ and $v:S\rightarrow \mathbb{R}$. 
Then we define a function space
concerned with weak solutions to \eqref{main.eq} by
\begin{equation*}
\mathcal{A}(\Omega) \coloneqq \left\{v:\mathbb{R}^{n}\rightarrow\mathbb{R}\ \Big|\ v |_{\Omega}\in L^p(\Omega)\ \text{ and } \  \iint_{\mathcal{C}_{\Omega}}H(x,y,|v(x)-v(y)|)\,\frac{dxdy}{|x-y|^{n}} < \infty \right\},
\end{equation*}
where $\mathcal C_\Omega$ is defined in \eqref{COmega}.
Note that $\varrho(v;\Omega)<\infty$ whenever $v \in \mathcal{A}(\Omega)$, which in particular implies
\begin{equation*}
\mathcal{A}(\Omega) \subset W^{s,p}(\Omega).
\end{equation*}
We note that if $sp > n$, then every function in $W^{s,p}(\Omega)$ is H\"older continuous by the fractional Sobolev embedding. Thus, in this paper we may assume without loss of generality that 
$$
sp \leq n.
$$
Moreover, again by the fractional Sobolev embedding, we have 
\begin{equation*}
\mathcal{A}(\Omega) \subset L^{q}(\Omega)
\quad\text{if }\ \begin{cases}
 p< q \leq  \frac{np}{n-sp} & \text{when }\ sp<n,\\  
p<q< \infty & \text{when }\ sp\ge n.
\end{cases}
\end{equation*}
This 
will be used later in the proof of several estimates concerning local boundedness.

We next define the tail space.
One of the 
features in \cite{DKP2} is to consider the notion of nonlocal tails in local estimates, which encodes the nonlocal nature of the problem.
We define 
\begin{align*}
L^{q-1}_{sp}(\mathbb{R}^{n}) \coloneqq \left\{v: \mathbb{R}^{n}\rightarrow \mathbb{R} \, \Big| \, \int_{\mathbb{R}^{n}}\frac{|v(x)|^{q-1}}{(1+|x|)^{n+sp}}\,dx < \infty  \right\}.
\end{align*}
Let $m \in \{s,t\}$ and $\ell \in \{p,q\}$. Since we have
\begin{align*}
\frac{1+|x|}{|x-x_{0}|} \leq \frac{1+|x-x_{0}|+|x_{0}|}{|x-x_{0}|} \leq 1+ \frac{1+|x_{0}|}{r},
\qquad
\frac{|v(x)|^{\ell-1}}{(1+|x|)^{n+m\ell}} \le \frac{|v(x)|^{q-1}+1}{(1+|x|)^{n+sp}}
\end{align*}
for $x \in \mathbb{R}^{n}\setminus B_{r}(x_{0})$,
we see that
\begin{equation*} 
\int_{\mathbb{R}^{n}\setminus B_{r}(x_{0})}\frac{|v(x)|^{\ell-1}}{|x-x_{0}|^{n+m\ell}}\,dx
\end{equation*}
is finite whenever $v \in L^{q-1}_{sp}(\mathbb{R}^{n})$ and $B_{r}(x_{0}) \subset \mathbb{R}^{n}$. 
We call such a quantity a nonlocal tail.

\begin{remark} 
If $v\in L^{q_{0}}(\mathbb{R}^{n})$ for some $q_{0}\geq q-1$, or if $v \in L^{q-1}(B_{R}(0))\cap L^{\infty}(\mathbb{R}^{n}\setminus B_{R}(0))$ for some $R>0$, then $v \in L^{q-1}_{sp}(\mathbb{R}^{n})$. Moreover, we have that
\[
 W^{s,p}(\R^n) \subset L^{q-1}_{sp}(\R^n) \quad \text{if }\ q\le p^*+1.
\]
\end{remark}

\subsection{Inequalities}
We first collect several inequalities 
concerning fractional Sobolev functions.
For basic properties of fractional Sobolev spaces, we refer to \cite{DPV}.

The following inequality shows an inclusion relation between fractional Sobolev spaces. We also notice that it fails to hold when $s=t$.
\begin{lemma}\label{inclusion}
Let $1 \leq p \leq q$ 
and $0<s<t<1$. Let $\Omega \subset \mathbb{R}^{n}$ be a bounded measurable set. Then, for any $v \in W^{t,q}(\Omega)$ we have
\begin{equation*}
\left(\int_{\Omega}\int_{\Omega}\frac{|v(x)-v(y)|^{p}}{|x-y|^{n+sp}}\,dxdy\right)^{\frac{1}{p}} \leq c|\Omega|^{\frac{q-p}{pq}}(\diam(\Omega))^{t-s}\left(\int_{\Omega}\int_{\Omega}\frac{|v(x)-v(y)|^q}{|x-y|^{n+tq}}\,dxdy\right)^{\frac{1}{q}}
\end{equation*}
for a constant $c\equiv c(n,s,t,p,q)$.
\end{lemma}
\begin{proof}
In the case $p<q$, this is a special case of \cite[Lemma 4.6]{coz}. In particular, the constant $c$ is given by
\begin{equation*}
c = \left(\frac{n(q-p)}{(t-s)pq}|B_{1}|\right)^{\frac{q-p}{pq}},
\end{equation*}
which blows up as $t \searrow s$. 

If $p=q$, then we directly have
\begin{align*}
\left(\int_{\Omega}\int_{\Omega}\frac{|v(x)-v(y)|^{p}}{|x-y|^{n+sp}}\,dxdy\right)^{\frac{1}{p}} & = \left(\int_{\Omega}\int_{\Omega}\frac{|v(x)-v(y)|^{p}}{|x-y|^{n+tp}}|x-y|^{(t-s)p}\,dxdy\right)^{\frac{1}{p}} \\
& \leq (\diam(\Omega))^{t-s}\left(\int_{\Omega}\int_{\Omega}\frac{|v(x)-v(y)|^{p}}{|x-y|^{n+tp}}\,dxdy\right)^{\frac{1}{p}}.
\end{align*}
\end{proof}

We next recall the following fractional Sobolev-Poincar\'e type inequalities from \cite[Lemma 2.5]{ok2021local}. In fact, the second one follows from the first one and Lemma \ref{inclusion}.
\begin{lemma}\label{SP}
Let $s\in(0,1)$, $p\geq 1$ be such that $sp \leq n$. For any $v \in W^{s,p}(B_{r})$ we have
\begin{equation*}
\left(\mean{B_{r}}|v-(v)_{B_{r}}|^{p_{s}^{*}}\,dx\right)^{\frac{p}{p_{s}^{*}}} \leq cr^{sp}\mean{B_{r}}\int_{B_{r}}\frac{|v(x)-v(y)|^{p}}{|x-y|^{n+sp}}\,dydx
\end{equation*}
for $c\equiv c(n,s,p)>0$. Moreover, if $s \le t<1$ and $p \le q$, we also have
\begin{equation*}
\left(\mean{B_{r}}|v-(v)_{B_{r}}|^{p_{s}^{*}}\,dx\right)^{\frac{q}{p_{s}^{*}}} \leq cr^{tq}\mean{B_{r}}\int_{B_{r}}\frac{|v(x)-v(y)|^{q}}{|x-y|^{n+tq}}\,dydx
\end{equation*}
for $c \equiv c(n,s,t,p,q)>0$, whenever the right-hand side is finite.
\end{lemma}


The following two lemmas are simple corollaries of the preceding lemma. They will be used in the proof of Theorems \ref{thm.bdd} and \ref{thm.hol}, respectively.

\begin{lemma}\label{ineq1}
Assume that the constants $s$, $t$, $p$ and $q$ satisfy \eqref{power} and \eqref{assumption.bdd}. Then for every $f \in W^{s,p}(B_{r})$ we have
\begin{align*}
\mean{B_{r}}\left|\frac{f}{r^{s}}\right|^{p} +L_0\left|\frac{f}{r^{t}}\right|^{q}\,dx 
&\leq c L_0 r^{(s-t)q}\left(\mean{B_{r}}\int_{B_{r}}\frac{|f(x)-f(y)|^{p}}{|x-y|^{n+sp}}\,dxdy\right)^{\frac{q}{p}} \\
& \quad + c\left(\frac{|\supp\, f|}{|B_{r}|}\right)^{\frac{sp}{n}}\mean{B_{r}}\int_{B_{r}}\frac{|f(x)-f(y)|^{p}}{|x-y|^{n+sp}}\,dxdy \\
&\quad + c\left(\frac{|\supp\, f|}{|B_{r}|}\right)^{p-1}\mean{B_{r}}\left|\frac{f}{r^{s}}\right|^{p} +L_0\left|\frac{f}{r^{t}}\right|^{q}\,dx
\end{align*}
for a constant $c\equiv c(n,s,t,p,q)$, where $L_0$ is any positive constant.
\end{lemma}
\begin{proof}
Applying H\"older's inequality and Lemma \ref{SP}, we have
\begin{align*}
\mean{B_{r}}\left|\frac{f}{r^{t}}\right|^{q}\,dx & \leq c\left(\mean{B_{r}}\left|\frac{f-(f)_{B_{r}}}{r^{t}}\right|^{p_{s}^{*}}\,dx\right)^{\frac{q}{p_{s}^{*}}} + c\left|\frac{(f)_{B_{r}}}{r^{t}}\right|^{q} \\
& \leq cr^{(s-t)q}\left(\mean{B_{r}}\int_{B_{r}}\frac{|f(x)-f(y)|^{p}}{|x-y|^{n+sp}}\,dxdy\right)^{\frac{q}{p}} + c\left|\frac{(f)_{B_{r}}}{r^{t}}\right|^{q}.
\end{align*}
Likewise, we obtain
\begin{align*}
\mean{B_{r}}\left|\frac{f}{r^{s}}\right|^{p}\,dx & \leq c\left(\frac{|\supp\, f|}{|B_{r}|}\right)^{\frac{sp}{n}}\left(\mean{B_{r}}\left|\frac{f-(f)_{B_{r}}}{r^{s}}\right|^{p_{s}^{*}}\,dx\right)^{\frac{p}{p_{s}^{*}}} + c\left|\frac{(f)_{B_{r}}}{r^{s}}\right|^{p} \\
& \leq c\left(\frac{|\supp\, f|}{|B_{r}|}\right)^{\frac{sp}{n}}\mean{B_{r}}\int_{B_{r}}\frac{|f(x)-f(y)|^{p}}{|x-y|^{n+sp}}\,dxdy + c\left|\frac{(f)_{B_{r}}}{r^{s}}\right|^{p}.
\end{align*}
We also have
\begin{align*}
\lefteqn{ \left|\frac{(f)_{B_{r}}}{r^{s}}\right|^{p} + L_0\left|\frac{(f)_{B_{r}}}{r^{t}}\right|^{q} } \\
& \leq r^{-sp}\left(\frac{|\supp\, f|}{|B_{r}|}\right)^{p-1}\mean{B_{r}}|f|^{p}\,dx + L_0r^{-tq}\left(\frac{|\supp\, f|}{|B_{r}|}\right)^{q-1}\mean{B_{r}}|f|^{q}\,dx \\
& \leq \left(\frac{|\supp\, f|}{|B_{r}|}\right)^{p-1}\mean{B_{r}}\left|\frac{f}{r^{s}}\right|^{p}+L_0\left|\frac{f}{r^{t}}\right|^{q}\,dx.
\end{align*}
We combine the inequalities in the above display to complete the proof.
\end{proof}

\begin{lemma}\label{ineq2}
Assume that the constants $s,t,p$ and $q$ satisfy \eqref{power} and that the function $a(\cdot,\cdot)$ satisfies \eqref{a.holder} and \eqref{assumption.hol}. Let $B_{r} \subseteq B_{R}$ be concentric balls with $ \frac{1}{2} R \leq r \leq R \leq 1$. 
Then for any $f \in L^{\infty}(B_{r})$ we have
\begin{align*}
\left[\mean{B_{r}}\left(\left|\frac{f}{r^{s}}\right|^{p}+a_2\left|\frac{f}{r^{t}}\right|^{q}\right)^{\kappa}\,dx\right]^{\frac{1}{\kappa}}
&\leq c\left(1+\lVert f \rVert_{L^{\infty}(B_{r})}^{q-p}\right)\mean{B_{r}}\int_{B_{r}}
H(x,y,|f(x)-f(y)|)\frac{dxdy}{|x-y|^{n}}\\
& \quad + c\left(1+\lVert f \rVert_{L^{\infty}(B_{r})}^{q-p}\right)\mean{B_{r}}\left|\frac{f}{r^{s}}\right|^{p} + a_{1}\left|\frac{f}{r^{t}}\right|^{q}\,dx
\end{align*}
for some $c \equiv c(n,s,t,p,q,[a]_{\alpha})>0$, whenever the right-hand side is finite, where 
\begin{equation*}
\kappa \coloneqq \min\left\{\frac{p_{s}^{*}}{p},\frac{q_{t}^{*}}{q}\right\}>1, \quad
a_{1} \coloneqq \inf_{B_{R}\times B_{R}}a(\cdot,\cdot)
\quad\text{and}\quad
a_{2} \coloneqq \sup_{B_{R}\times B_{R}}a(\cdot,\cdot).
\end{equation*}  
\end{lemma}
\begin{proof}
Using the assumptions, we estimate 
\begin{align*}
\lefteqn{ \left[\mean{B_{r}}\left(\left|\frac{f}{r^{s}}\right|^{p}+a_2\left|\frac{f}{r^{t}}\right|^{q}\right)^{\kappa}\,dx\right]^{\frac{1}{\kappa}} } \\
& \leq c\left[\mean{B_{r}}\left(\left|\frac{f}{r^{s}}\right|^{p} + a_{1}\left|\frac{f}{r^{t}}\right|^{q}\right)^{\kappa} + \left(r^{\alpha+sp-tq}\lVert f \rVert_{L^{\infty}(B_{r})}^{q-p}\left|\frac{f}{r^{s}}\right|^{p}\right)^{\kappa}\,dx \right]^{\frac{1}{\kappa}} \\
& \leq c\left(1+\lVert f \rVert_{L^{\infty}(B_{r})}^{q-p}\right)\left[\mean{B_{r}}\left(\left|\frac{f}{r^{s}}\right|^{p}+a_{1}\left|\frac{f}{r^{t}}\right|^{q}\right)^{\kappa}\,dx\right]^{\frac{1}{\kappa}}.
\end{align*}
We next apply Lemma \ref{SP} to see that
\begin{align*}
\lefteqn{\left[\mean{B_{r}}\left(\left|\frac{f}{r^{s}}\right|^{p}+a_{1}\left|\frac{f}{r^{t}}\right|^{q}\right)^{\kappa}\,dx\right]^{\frac{1}{\kappa}}}\\
& \leq c\left[\mean{B_{r}}\left(\left|\frac{f-(f)_{B_{r}}}{r^{s}}\right|^{p}+a_{1}\left|\frac{f-(f)_{B_{r}}}{r^{t}}\right|^{q}\right)^{\kappa}\,dx\right]^{\frac{1}{\kappa}} + c\left|\frac{(f)_{B_{r}}}{r^{s}}\right|^{p} + ca_{1}\left|\frac{(f)_{B_{r}}}{r^{t}}\right|^{q} \\
& \leq c\mean{B_{r}}\int_{B_{r}}\frac{|f(x)-f(y)|^{p}}{|x-y|^{n+sp}}+a_{1}\frac{|f(x)-f(y)|^{q}}{|x-y|^{n+tq}}\,dxdy + c\mean{B_{r}}\left|\frac{f}{r^{s}}\right|^{p} + a_{1}\left|\frac{f}{r^{t}}\right|^{q}\,dx \\
&\leq c\mean{B_{r}}\int_{B_{r}}
H(x,y,|f(x)-f(y)|)\frac{dxdy}{|x-y|^{n}}
+c\mean{B_{r}}\left|\frac{f}{r^{s}}\right|^{p} + a_{1}\left|\frac{f}{r^{t}}\right|^{q}\,dx.
\end{align*}
Then the conclusion follows.
\end{proof}

The following numerical inequalities will be frequently used in this paper.

\begin{lemma}\label{size.comp}
Let $p\ge 1$ and $a,b \geq 0$. Then 
\begin{equation*}
a^{p} - b^{p} \leq pa^{p-1}|a-b|
\end{equation*}
and, for any $\varepsilon > 0$,
\begin{equation*}
a^{p} - b^{p} \leq \varepsilon b^{p} + c\varepsilon^{1-p}|a-b|^{p}
\end{equation*}
for some $c \equiv c(p) > 0$.
\end{lemma}
\begin{proof}
The first one is a direct consequence of 
Mean Value Theorem; note that we may assume $a \geq b$, otherwise it is obvious. For the proof of the second one, see \cite[Lemma 3.1]{DKP2}.
\end{proof}

We end this section with a standard iteration lemma from \cite[Lemma 7.1]{MR1962933}.
\begin{lemma}\label{iterlem} 
Let $\{y_{i}\}_{i=0}^{\infty}$ be a sequence of nonnegative numbers satisfying
\begin{equation*}
y_{i+1} \leq b_{1}b_{2}^{i}y_{i}^{1+\beta}, \qquad i=0,1,2,...,
\end{equation*}
for some constants $b_{1},\beta>0$ and $b_{2}>1$. If
\begin{equation*}
y_{0} \leq b_{1}^{-1/\beta}b_{2}^{-1/\beta^{2}},
\end{equation*}
then $y_{i} \rightarrow 0$ as $i\rightarrow\infty$.
\end{lemma}

\section{Existence of weak solutions}\label{existence}
In this section we show the existence of weak solutions to \eqref{main.eq}. By a standard argument, such as the one in the proof of \cite[Theorem 2.3]{DKP2}, we see that $u \in \mathcal{A}(\Omega)$ is a weak solution to \eqref{main.eq}
if and only if it is a minimizer of the functional
\begin{equation}\label{functional}
\mathcal{E}(v;\Omega) \coloneqq \iint_{\mathcal{C}_{\Omega}}\frac{1}{p}|v(x)-v(y)|^{p}K_{sp}(x,y) + a(x,y)\frac{1}{q}|v(x)-v(y)|^{q}K_{tq}(x,y)\,dxdy.
\end{equation}
We say that $u \in \mathcal{A}(\Omega)$ is a minimizer of \eqref{functional} if
\begin{equation*}
\mathcal{E}(u;\Omega) \leq \mathcal{E}(v;\Omega)
\end{equation*}
for every $v \in \mathcal{A}(\Omega)$ with $v=u$ a.e. in $\mathbb{R}^{n}\setminus\Omega$. Therefore, we prove the existence and uniqueness of the minimizer of \eqref{functional} with a Dirichlet boundary condition. 
\begin{theorem}
Let $\Omega$ be a bounded domain and $g \in \mathcal{A}(\Omega)$ be a given boundary data. Let $K_{sp},K_{tq},a:\mathbb{R}^{n}\times\mathbb{R}^{n}\rightarrow\mathbb{R}$ be symmetric and satisfy \eqref{kernel.growth}-\eqref{a.bound}. Then there exists a unique minimizer $u \in \mathcal{A}(\Omega)$ of \eqref{functional} with $u=g$ a.e. in $\mathbb{R}^{n}\setminus\Omega$. Moreover, if $g \in \mathcal{A}(\Omega) \cap L^{q-1}_{sp}(\mathbb{R}^{n})$, then $u \in \mathcal{A}(\Omega) \cap L^{q-1}_{sp}(\mathbb{R}^{n})$.
\end{theorem}
\begin{proof}
The uniqueness follows directly from the fact that the function $\tau \mapsto \tau^{p} + a(x,y) \tau^{q}$ is  strictly convex 
for each fixed $(x,y)$. Now we prove the existence. The admissible set 
$$\mathcal{A}_g(\Omega) \coloneqq \{v\in\mathcal{A}(\Omega):v=g \text{ a.e. in }\R^n\setminus \Omega\}$$ 
is obviously nonempty, as $g \in \mathcal{A}_g(\Omega)$. 
Let $\{u_{k}\}_{k} \subset \mathcal{A}_g(\Omega)$ be a minimizing sequence. Then there exists a constant $C$ such that 
\begin{equation*}
[u_{k}]_{s,p;\Omega}^{p} = \int_{\Omega}\int_{\Omega}\frac{|u_{k}(x)-u_{k}(y)|^{p}}{|x-y|^{n+sp}}\,dxdy \leq \mathcal{E}(u_{k};\Omega) \leq C \qquad \forall\, k \in \mathbb{N}.
\end{equation*}
In particular, Lemma~\ref{inclusion}
implies that $\{[u_{k}]_{s_{0},p;\Omega}\}_{k}$ is bounded for any $s_{0} \in (0,s)$. 
Then we choose a ball $B_{R}\equiv B_{R}(x_{0}) \supset \Omega$ with $R \geq 1$ and fix $s_{0} \in (0,s/2)$ with $np/(n+s_{0}p) \eqqcolon p_{0} > 1$. Since $u_{k}-g = 0$ a.e. in $\mathbb{R}^{n}\setminus \Omega$, the fractional Sobolev embedding \cite[Theorem 6.5]{DPV} implies
\begin{equation}\label{ffff}
\begin{aligned}
&\left(\int_{B_{R}}|u_{k}-g|^{p}\,dx\right)^{\frac{p_{0}}{p}} 
\leq c[u_{k}-g]_{s_{0},p_{0};\mathbb{R}^{n}}^{p_{0}} \\
& \leq c[u_{k}-g]_{s_{0},p_{0};B_{R}}^{p_{0}}  + c\int_{B_{R}}|u_{k}(y)-g(y)|^{p_{0}}\left(\int_{\mathbb{R}^{n}\setminus B_{R}}\frac{dx}{|x-y|^{n+s_{0}p_{0}}}\right)\,dy  \\
& \leq c[u_{k}-g]_{s_{0},p_{0};B_{R}}^{p_{0}}  + c\int_{B_{R}}|u_{k}(y)-g(y)|^{p_{0}}\left(\int_{B_{2R}\setminus B_{R}}\frac{dx}{|x-y|^{n+s_{0}p_{0}}}\right)\,dy  \\
& \leq c\int_{B_{2R}}\int_{B_{R}}\frac{|(u_{k}-g)(x)-(u_{k}-g)(y)|^{p_{0}}}{|x-y|^{n+s_{0}p_{0}}}\,dxdy 
\leq c[u_{k}-g]_{s_{0},p_{0};B_{2R}}^{p_{0}},
\end{aligned}
\end{equation}
where we have used the fact that
\begin{equation*}
\int_{\mathbb{R}^{n}\setminus B_{R}}\frac{dx}{|x-y|^{n+s_{0}p_{0}}} \leq c(n)\int_{B_{2R}\setminus B_{R}}\frac{dx}{|x-y|^{n+s_{0}p_{0}}} \qquad \forall\, y \in B_{R}.
\end{equation*}
Applying Lemma \ref{inclusion} to the right-hand side of \eqref{ffff}, we have for all $k \in \mathbb{N}$
\begin{align*}
 \left(\int_{B_{R}}|u_{k}-g|^{p}\,dx\right)^{\frac{p_{0}}{p}} & \leq c[u_{k}-g]_{s_{0},p_{0};B_{2R}}^{p_{0}} \\
& \leq cR^{sp_{0}}[u_{k}-g]_{s,p;B_{2R}}^{p_{0}} \\
& \leq cR^{sp_{0}}\left(\iint_{\mathcal{C}_{\Omega}}\frac{|u_{k}(x)-u_{k}(y)|^{p}}{|x-y|^{n+sp}}\,dxdy + \iint_{\mathcal{C}_{\Omega}}\frac{|g(x)-g(y)|^{p}}{|x-y|^{n+sp}}\,dxdy\right)^{\frac{p_{0}}{p}} \\
& \leq cR^{sp_{0}}\left(C + \iint_{\mathcal{C}_{\Omega}}\frac{|g(x)-g(y)|^{p}}{|x-y|^{n+sp}}\,dxdy\right)^{\frac{p_{0}}{p}}.
\end{align*}
This implies that $\{u_{k}-g\}_{k}$ is bounded in $L^{p}(B_{R})$, and hence in $W^{s,p}_{0}(B_{R})$.
By the compact embedding theorem for fractional Sobolev spaces \cite[Theorem 7.1]{DPV}, there exist a subsequence $\{u_{k_{j}}-g\}_{j}$ and $v \in L^{p}(B_R)$ such that
\begin{align*}
\begin{cases}
u_{k_{j}}-g\ \longrightarrow\ v \quad \textrm{in} \quad L^{p}(B_{R}), \\
u_{k_{j}}-g\ \longrightarrow\ v \quad \textrm{a.e. in} \;\; B_{R},
\end{cases}
\quad \text{as }\ j\ \to\ \infty.
\end{align*}
We extend $v$ to $\mathbb{R}^{n}$ by letting $v = 0$ on $\mathbb{R}^{n}\setminus B_R$ and set $u \coloneqq v+g$. Then $u_{k_{j}} \rightarrow u$ a.e. in $\mathbb{R}^{n}$.
Finally, Fatou's lemma implies
\begin{align*}
\lefteqn{ \iint_{\mathcal{C}_{\Omega}}\frac{1}{p}|u(x)-u(y)|^{p}K_{sp}(x,y) + a(x,y)\frac{1}{q}|u(x)-u(y)|^{q}K_{tq}(x,y)\,dxdy }\\
& \leq \liminf_{j\rightarrow\infty} \iint_{\mathcal{C}_{\Omega}}\frac{1}{p}|u_{k_j}(x)-u_{k_j}(y)|^{p}K_{sp}(x,y) + a(x,y)\frac{1}{q}|u_{k_j}(x)-u_{k_j}(y)|^{q}K_{tq}(x,y)\,dxdy.
\end{align*}
This means that $u \in \mathcal{A}_g(\Omega)$ and it is a minimizer of $\mathcal{E}$.


The last assertion is clear since $u\in L^p(\Omega)$ and $u=g$ a.e. in $\R^n\setminus\Omega$.
\end{proof}

\begin{remark}
In fact, the above theorem still holds even when $a(\cdot,\cdot) \geq 0$ is not bounded above.
\end{remark}

\section{Caccioppoli estimates and Local boundedness}
\label{Caccio.bounded}
  
We start with the following lemma which implies that the multiplication of any function in $\mathcal{A}(\Omega)$ and a cut-off function is also a function in $\mathcal{A}(\Omega)$.  We recall the notation \eqref{modular}.
\begin{lemma}\label{lem.test}
Assume that the constants $s$, $t$, $p$ and $q$ satisfy \eqref{power}, and $\eta \in W^{1,\infty}_{0}(B_{r})$. If one of the following two conditions holds: 
\begin{itemize}
\item[(i)] The inequality  \eqref{assumption.bdd} holds and $v \in L^{p}(B_{2r})$ satisfies $\varrho(v;B_{2r})<\infty$;
\item[(ii)] $v \in L^{q}(B_{2r})$ satisfies $\varrho(v;B_{2r})<\infty$,
\end{itemize}
then $\varrho(v\eta;\mathbb{R}^{n})<\infty$. In particular, $v\eta \in \mathcal{A}(\Omega)$ whenever $\Omega \supset B_{2r}$.
\end{lemma}
\begin{proof}
Note that we also have $v \in L^q(B_{2r})$ in (i) by Lemma~\ref{SP} and \eqref{assumption.bdd}. We write
\begin{equation*}
\varrho(v\eta;\mathbb{R}^{n}) = \varrho(v\eta;B_{2r}) + 2\int_{\mathbb{R}^{n}\setminus B_{2r}}\int_{B_{2r}}H(x,y,|v(x)\eta(x)|)\,\frac{dxdy}{|x-y|^{n}}.
\end{equation*}
The first term is estimated as
\begin{align*}
&\varrho(v\eta;B_{2r})\\
& \leq c\int_{B_{2r}}\int_{B_{2r}}H(x,y,|(v(x)-v(y))\eta(y)|)\,\frac{dxdy}{|x-y|^{n}} \\
& \quad + c\int_{B_{2r}}\int_{B_{2r}}H(x,y,|v(x)(\eta(x)-\eta(y))|)\,\frac{dxdy}{|x-y|^{n}} \\
& \leq c\left(\lVert \eta \rVert_{L^{\infty}(B_{2r})} + 1\right)^{q}\varrho(v;B_{2r}) + c \lVert D\eta \rVert_{L^{\infty}(B_{2r})}^{p}\int_{B_{2r}}|v(x)|^{p}\int_{B_{4r}(x)}\frac{dy}{|x-y|^{n+(s-1)p}}\,dx \\
& \quad + c \lVert D\eta \rVert_{L^{\infty}(B_{2r})}^{q} \lVert a \rVert_{L^{\infty}}\int_{B_{2r}}|v(x)|^{q}\int_{B_{4r}(x)}\frac{dy}{|x-y|^{n+(t-1)q}}\,dx \\
& \leq c\left(\lVert \eta \rVert_{L^{\infty}(B_{2r})} + 1\right)^{q}\varrho(v;B_{2r}) +c\lVert D\eta \rVert_{L^{\infty}(B_{2r})}^pr^{(1-s)p}\int_{B_{2r}}|v(x)|^{p}\,dx \\
& \quad + c\lVert D\eta \rVert_{L^{\infty}(B_{2r})}^{q} \lVert a \rVert_{L^{\infty}}r^{(1-t)q}\int_{B_{2r}}|v(x)|^{q}\,dx \\
& < \infty.
\end{align*}
The second term is estimated as
\begin{align*}
\lefteqn{ \int_{\mathbb{R}^{n}\setminus B_{2r}}\int_{B_{2r}}H(x,y,|v(x)\eta(x)|)\,\frac{dxdy}{|x-y|^{n}} } \\
& \leq \left(\lVert \eta \rVert_{L^{\infty}(B_{2r})} + 1\right)^{q}\int_{\mathbb{R}^{n}\setminus B_{2r}}\int_{B_{r}}\frac{|v(x)|^{p}}{|x-y|^{n+sp}} + \lVert a \rVert_{L^{\infty}}\frac{|v(x)|^{q}}{|x-y|^{n+tq}}\,dxdy \\
& \leq c\left(\lVert \eta \rVert_{L^{\infty}(B_{2r})} + 1\right)^{q}\left( r^{-sp}\int_{B_{r}}|v(x)|^{p}\,dx + \lVert a \rVert_{L^{\infty}}r^{-tq}\int_{B_{r}}|v(x)|^{q}\,dx \right)  < \infty,
\end{align*}
and  the conclusion follows.
\end{proof}

Next, we prove a nonlocal Caccioppoli type estimate. We again recall \eqref{modular} with \eqref{def.H}, and further define
\begin{equation}\label{def.h}
h(x,y,\tau) \coloneqq \frac{\tau^{p-1}}{|x-y|^{sp}}+a(x,y) \frac{\tau^{q-1}}{|x-y|^{tq}}, \quad x,y\in\mathbb{R}^{n}, \tau \geq 0.
\end{equation}

\begin{lemma}\label{caccioppoli.est}
Let $K_{sp},K_{tq},a:\mathbb{R}^{n}\times\mathbb{R}^{n}\rightarrow\mathbb{R}$ be symmetric and satisfy \eqref{kernel.growth}-\eqref{a.bound} with \eqref{assumption.bdd}. Let $u\in\mathcal{A}(\Omega)\cap L^{q-1}_{sp}(\Omega)$ be a weak solution to \eqref{main.eq}. Then for any ball $B_{2r}\equiv B_{2r}(x_{0}) \Subset \Omega$ and any $\phi \in C^{\infty}_{0}(B_{r})$ with $0\leq\phi\leq1$, we have
\begin{equation}\label{caccio}\begin{aligned}
\lefteqn{ \int_{B_{r}}\int_{B_{r}}H(x,y,|w_{\pm}(x)-w_{\pm}(y)|)(\phi^{q}(x)+\phi^{q}(y))\,\frac{dxdy}{|x-y|^{n}} }\\
& \leq c\int_{B_{r}}\int_{B_{r}}H\big(x,y,|(\phi(x)-\phi(y))(w_{\pm}(x)+w_{\pm}(y))|\big)\,\frac{dxdy}{|x-y|^{n}} \\
& \quad + c\left(\sup_{y \in \supp\,\phi}\int_{\mathbb{R}^{n}\setminus B_{r}}h(x,y,w_{\pm}(x))\,\frac{dx}{|x-y|^{n}}\right)\int_{B_{r}}w_{\pm}(x)\phi^{q}(x)\,dx
\end{aligned}\end{equation}
for some $c\equiv c(n,s,t,p,q,\Lambda)>0$, where $w_{\pm} \coloneqq (u-k)_{\pm}$ with $k\ge 0$.
\end{lemma}
\begin{proof}
We follow the proof of \cite[Lemma 3.1]{giacomoni2021sup} and only consider the estimate for $w_{+}$, since the estimate for $w_{-}$ can be proved similarly. In light of Lemma~\ref{lem.test} for the case (i), we can test the weak formulation \eqref{weak.formulation} with $w_{+}\phi^{q}\in \mathcal{A}(\Omega)$. Using the short notation 
\begin{equation}\label{phiell}
\Phi_{\ell}(\tau) \coloneqq |\tau|^{\ell-2}\tau , \quad \ell \in \{p,q\} \ \text{ and } \ \tau \geq 0,
\end{equation}
 we have
\begin{align*}
0 & = \int_{B_{r}}\int_{B_{r}}\left[\Phi_{p}(u(x)-u(y))(w_{+}(x)\phi^{q}(x)-w_{+}(y)\phi^{q}(y))K_{sp}(x,y)\right. \\
& \qquad \qquad \quad + \left. a(x,y)\Phi_{q}(u(x)-u(y))(w_{+}(x)\phi^{q}(x)-w_{+}(y)\phi^{q}(y))K_{tq}(x,y)\right]\,dxdy \\
& \quad + 2\int_{\mathbb{R}^{n}\setminus B_{r}}\int_{B_{r}}\left[\Phi_{p}(u(x)-u(y))w_{+}(x)\phi^{q}(x)K_{sp}(x,y)\right. \\
& \qquad \qquad \qquad \qquad + \left. a(x,y)\Phi_{q}(u(x)-u(y))w_{+}(x)\phi^{q}(x)K_{tq}(x,y)\right]\,dxdy \\
& \eqqcolon I_{1} + I_{2}.
\end{align*}

We first estimate $I_1$. Assume that  $u(x) \geq u(y)$. Then,
\begin{align*}
\lefteqn{\Phi_{\ell}(u(x)-u(y))(w_{+}(x)\phi^{q}(x)-w_{+}(y)\phi^{q}(y))}\\
& = (u(x)-u(y))^{\ell-1}((u(x)-k)_{+}\phi^{q}(x)-(u(y)-k)_{+}\phi^{q}(y)) \\
& = \begin{cases}
(w_{+}(x)-w_{+}(y))^{\ell-1}(w_{+}(x)\phi^{q}(x)-w_{+}(y)\phi^{q}(y)), \quad & u(x)\geq u(y) \geq k\\
(u(x)-u(y))^{\ell-1}w_{+}(x)\phi^{q}(x), & u(x) > k \geq u(y) \\
0, & k\geq u(x)\geq u(y)
\end{cases} \\
& \geq (w_{+}(x)-w_{+}(y))^{\ell-1}(w_{+}(x)\phi^{q}(x)-w_{+}(y)\phi^{q}(y)) \\
& = \Phi_{\ell}(w_{+}(x)-w_{+}(y))(w_{+}(x)\phi^{q}(x)-w_{+}(y)\phi^{q}(y)),
\end{align*}
and hence
\begin{equation}\label{I1.est}
\begin{aligned}
I_{1} & \geq \int_{B_{r}}\int_{B_{r}}\left[\Phi_{p}(w_{+}(x)-w_{+}(y))(w_{+}(x)\phi^{q}(x)-w_{+}(y)\phi^{q}(y))K_{sp}(x,y)\right. \\
 & \qquad \qquad \quad + \left. a(x,y)\Phi_{q}(w_{+}(x)-w_{+}(y))(w_{+}(x)\phi^{q}(x)-w_{+}(y)\phi^{q}(y))K_{tq}(x,y)\right]\,dxdy.
\end{aligned}
\end{equation}
Moreover, 
\begin{align*}
w_{+}(x)\phi^{q}(x)-w_{+}(y)\phi^{q}(y) = \frac{w_{+}(x)-w_{+}(y)}{2}(\phi^{q}(x)+\phi^{q}(y))+\frac{w_{+}(x)+w_{+}(y)}{2}(\phi^{q}(x)-\phi^{q}(y)),
\end{align*}
which implies
\begin{align*}
\lefteqn{ \Phi_{\ell}(w_{+}(x)-w_{+}(y))(w_{+}(x)\phi^{q}(x)-w_{+}(y)\phi^{q}(y)) }\\ 
& \geq |w_{+}(x)-w_{+}(y)|^{\ell}\frac{\phi^{q}(x)+\phi^{q}(y)}{2} - |w_{+}(x)-w_{+}(y)|^{\ell-1}\frac{w_{+}(x)+w_{+}(y)}{2}|\phi^{q}(x)-\phi^{q}(y)|.
\end{align*}
Here, we use Lemma \ref{size.comp} to see that
\begin{equation*}
\begin{aligned}
|\phi^{q}(x)-\phi^{q}(y)| & \leq q(\phi^{q-1}(x)+\phi^{q-1}(y))|\phi(x)-\phi(y)| \\
& \leq c(q)(\phi^{q}(x)+\phi^{q}(y))^{(q-1)/q}|\phi(x)-\phi(y)|.
\end{aligned}
\end{equation*}
Thus, using Young's inequality, we get
\begin{align*}
\lefteqn{ |w_{+}(x)-w_{+}(y)|^{\ell-1}(w_{+}(x)+w_{+}(y))|\phi^{q}(x)-\phi^{q}(y)| } \\
& \leq |w_{+}(x)-w_{+}(y)|^{\ell-1}(w_{+}(x)+w_{+}(y))(\phi^{q}(x)+\phi^{q}(y))^{\frac{\ell-1}{\ell}+\frac{q-\ell}{q\ell}}|\phi(x)-\phi(y)| \\
& \leq \varepsilon|w_{+}(x)-w_{+}(y)|^{\ell}(\phi^{q}(x)+\phi^{q}(y)) \\
& \quad + c(\varepsilon)(\phi^{q}(x)+\phi^{q}(y))^{(q-\ell)/q}|\phi(x)-\phi(y)|^{\ell}(w_{+}(x)+w_{+}(y))^{\ell}.
\end{align*}
Since $0 \leq \phi \leq 1$ and $(q-\ell)/q \geq 0$, after choosing $\varepsilon$ so small, we discover
\begin{align*}
\lefteqn{\Phi_{\ell}(w_{+}(x)-w_{+}(y))(w_{+}(x)\phi^{q}(x)-w_{+}(y)\phi^{q}(y))} \\
& \geq |w_{+}(x)-w_{+}(y)|^{\ell}\frac{\phi^{q}(x)+\phi^{q}(y)}{4} - c|\phi(x)-\phi(y)|^{\ell}(w_{+}(x)+w_{+}(y))^{\ell}.
\end{align*}
We notice that by the symmetry of  the above inequality for $x$ and $y$, we also have the same inequality when $u(x)<u(y)$.
Inserting this into \eqref{I1.est} and using \eqref{kernel.growth}, we have 
\begin{align*}
I_{1} & \geq  \frac{1}{4\Lambda}\int_{B_{r}}\int_{B_{r}}H(x,y,|w_{+}(x)-w_{+}(y)|)(\phi^{q}(x)+\phi^{q}(y))\,\frac{dxdy}{|x-y|^{n}} \\
&\quad -  c\int_{B_{r}}\int_{B_{r}}H(x,y, |\phi(x)-\phi(y)|(w_{+}(x)+w_{+}(y)))\,\frac{dxdy}{|x-y|^{n}} .
\end{align*}

For $I_{2}$, we observe that 
\begin{align*}
\Phi_{\ell}(u(x)-u(y))w_{+}(x) \geq -w_{+}^{\ell-1}(y)w_{+}(x)
\end{align*}
and use \eqref{kernel.growth}, to find 
\begin{align*}
I_{2}
& \geq -c\int_{\mathbb{R}^{n}\setminus B_{r}}\int_{B_{r}}h(x,y,w_{+}(y)) w_{+}(x)\phi^{q}(x)\,\frac{dxdy}{|x-y|^{n}} \\
& \geq -c\left(\sup_{x\in \supp\,\phi}\int_{\mathbb{R}^{n}\setminus B_{r}}h(x,y,w_{+}(y))\,\frac{dy}{|x-y|^{n}}\right)\int_{B_{r}}w_{+}(x)\phi^{q}(x)\,dx.
\end{align*}
Combining the above 
estimates with $I_{1}+I_{2}=0$, we obtain \eqref{caccio}.
\end{proof}

\begin{remark}\label{rmk.caccioppoli}
If $u$ is locally bounded, then the assumption \eqref{assumption.bdd} in Lemma \ref{caccioppoli.est} can be eliminated,
see the case (ii) in Lemma \ref{lem.test}.
Moreover, we can also obtain \eqref{caccio} when $q>p^*_s$ by using a truncation argument as in \cite[Lemma 4.2]{ok2021local} provided the right-hand side of \eqref{caccio} is finite.
\end{remark}

Now, we are ready to  prove the local boundedness of weak solutions to \eqref{main.eq}. 

\begin{proof}[Proof of Theorem \ref{thm.bdd}]
For convenience, we define
\begin{align*}
H_{0}(\tau) &\coloneqq \tau^{p}+\lVert a \rVert_{L^{\infty}}\tau^{q}, \quad \tau\ge 0.
\end{align*}
In the following, $c$ means a constant depending only on $\data$.

Let $B_{r} \equiv B_{r}(x_{0}) \Subset \Omega$ be a fixed ball with $r \leq 1$. For $r/2 \leq \rho < \sigma \leq r$ and $k>0$, we denote 
\begin{equation*}
A^{+}(k,\rho) \coloneqq \{x\in B_{\rho}: u(x) \geq k \}
\end{equation*} 
and apply Lemma \ref{ineq1} with $f \equiv (u-k)_{+}$ to have
\begin{equation}\label{sf}
\begin{aligned}
\rho^{-sp}\mean{B_{\rho}}H_{0}(f)\,dx
& \leq \mean{B_{\rho}}\left(\frac{f}{\rho^{s}}\right)^{p}+\lVert a \rVert_{L^{\infty}}\left(\frac{f}{\rho^{t}}\right)^{q}\,dx \\
& \leq c\lVert a \rVert_{L^{\infty}}\rho^{(s-t)q}\left(\mean{B_{\rho}}\int_{B_{\rho}}\frac{|f(x)-f(y)|^{p}}{|x-y|^{n+sp}}\,dxdy\right)^{\frac{q}{p}} \\
& \quad + c\left(\frac{|A^{+}(k,\rho)|}{|B_{\rho}|}\right)^{\frac{sp}{n}}\mean{B_{\rho}}\int_{B_{\rho}}\frac{|f(x)-f(y)|^{p}}{|x-y|^{n+sp}}\,dxdy \\
& \quad + c\left(\frac{|A^{+}(k,\rho)|}{|B_{\rho}|}\right)^{p-1}\mean{B_{\sigma}}\left(\frac{f}{\rho^{s}}\right)^{p} + \lVert a \rVert_{L^{\infty}}\left(\frac{f}{\rho^{t}}\right)^{q}\,dx.
\end{aligned}
\end{equation}
We now fix $0<h<k$ and observe that, for $x\in A^{+}(k,\rho) \subset A^{+}(h,\rho)$, 
\begin{align*}
&(u(x)-h)_{+} = u(x)-h \geq k-h, \\
&(u(x)-h)_{+} = u(x)-h \geq u(x)-k = (u(x)-k)_{+}.
\end{align*}
This implies
\begin{equation}\label{aplus.est}
|A^{+}(k,\rho)| \leq \int_{A^{+}(k,\rho)}\frac{(u-h)_{+}^{p}}{(k-h)^{p}}\,dx \leq \frac{1}{(k-h)^{p}}\int_{A^{+}(h,\sigma)}H_{0}((u-h)_{+})\,dx
\end{equation}
and
\begin{equation}\label{uk.est}
\begin{aligned}
\mean{B_{\rho}}(u-k)_{+}\,dx &\leq \mean{B_{\sigma}}(u-h)_{+}\left(\frac{(u-h)_{+}}{k-h}\right)^{p-1}\,dx \\
& \leq \frac{1}{(k-h)^{p-1}}\mean{B_{\sigma}}H_{0}((u-h)_{+})\,dx.
\end{aligned}
\end{equation}
We then choose a cut-off function $\phi \in C^{\infty}_{0}(B_{\frac{\rho+\sigma}{2}})$ satisfying $0\leq\phi\leq1$, $\phi \equiv 1$ in $B_{\rho}$ and $|D\phi| \leq 4/(\sigma-\rho)$. 
Denoting the tail by
\begin{equation*}
T(v;r) \coloneqq \int_{\mathbb{R}^{n}\setminus B_{r}}\frac{|v(x)|^{p-1}}{|x-x_{0}|^{n+sp}}+\lVert a \rVert_{L^{\infty}}\frac{|v(x)|^{q-1}}{|x-x_{0}|^{n+sp}}\,dx,
\end{equation*} 
Lemma~\ref{caccioppoli.est} gives
\begin{equation*}
\begin{aligned}
\lefteqn{ \mean{B_{\rho}}\int_{B_{\rho}}\frac{|f(x)-f(y)|^{p}}{|x-y|^{n+sp}}\,dxdy }\\
& \leq \frac{c}{(\sigma-\rho)^{p}}\mean{B_{\sigma}}(u(x)-h)_{+}^{p}\int_{B_{\sigma}}\frac{1}{|x-y|^{n+(s-1)p}}\,dydx \\
& \quad + \frac{c\lVert a \rVert_{L^{\infty}}}{(\sigma-\rho)^{q}}\mean{B_{\sigma}}(u(x)-h)_{+}^{q}\int_{B_{\sigma}}\frac{1}{|x-y|^{n+(t-1)q}}\,dydx \\
& \quad +c\left(\sup_{x\in\supp\,\phi}\int_{\mathbb{R}^{n}\setminus B_{\sigma}}\frac{(u(y)-k)_{+}^{p-1}}{|x-y|^{n+sp}}+\lVert a \rVert_{L^{\infty}}\frac{(u(y)-k)_{+}^{q-1}}{|x-y|^{n+tq}}\,dy\right)\mean{B_{\sigma}}(u(x)-k)_{+}\,dx \\
& \leq \frac{c\rho^{(1-s)p}}{(\sigma-\rho)^{p}}\mean{B_{\sigma}}(u-h)_{+}^{p}\,dx + \frac{c\lVert a \rVert_{L^{\infty}}\rho^{(1-t)q}}{(\sigma-\rho)^{q}}\mean{B_{\sigma}}(u-h)_{+}^{q}\,dx \\
& \quad + \frac{c(\sigma+\rho)}{\sigma-\rho}\left(\int_{\mathbb{R}^{n}\setminus B_{\sigma}}\frac{(u(y)-k)_{+}^{p-1}}{|y-x_{0}|^{n+sp}}+\lVert a \rVert_{L^{\infty}}\frac{(u(y)-k)_{+}^{q-1}}{|y-x_{0}|^{n+tq}}\,dy\right)\mean{B_{\tau}}(u(x)-k)_{+}\,dx \\
& \leq \frac{cr^{(1-t)p}}{(\sigma-\rho)^{q}}\mean{B_{\sigma}}H_{0}((u-h)_{+})\,dx + \frac{c(\sigma+\rho)}{\sigma-\rho}\left[T((u-k)_{+};\sigma)\right]\mean{B_{\tau}}(u-k)_{+}\,dx,
\end{aligned}
\end{equation*}
where 
we have used that
$\frac{|y-x_{0}|}{|y-x|} \leq 1+\frac{|x-x_{0}|}{|y-x|} \leq 1+\frac{\sigma+\rho}{\sigma-\rho} \leq  2\frac{\sigma+\rho}{\sigma-\rho}$ for $x\in \supp\, \phi$ and $y\in \R^n\setminus B_{\sigma}$.
Combining this estimate together with \eqref{sf}-\eqref{uk.est} implies 
\begin{align*}
\lefteqn{\rho^{-sp}\mean{B_{\rho}}H_{0}((u-k)_{+})\,dx}\\
& \leq c\rho^{(s-t)q}\frac{r^{(1-t)q}}{(\sigma-\rho)^{q^{2}/p}}\left(\mean{B_{\sigma}}H_{0}((u-h)_{+})\,dx\right)^{\frac{q}{p}} \\
& \quad + c\frac{\sigma+\rho}{\sigma-\rho}\rho^{(s-t)q}[T((u-k)_{+};\sigma)]^{\frac{q}{p}}\frac{1}{(k-h)^{q/p'}}\left(\mean{B_{\sigma}}H_{0}((u-h)_{+})\,dx\right)^{\frac{q}{p}} \\
& \quad + \frac{c}{(k-h)^{sp^{2}/n}}\frac{r^{(1-t)p}}{(\sigma-\rho)^{q}}\left(\mean{B_{\sigma}}H_{0}((u-h)_{+})\,dx\right)^{1+\frac{sp}{n}} \\
& \quad + \frac{c(\sigma+\rho)(k-h)^{sp^{2}/n+p-1}}{\sigma-\rho} [T((u-k)_{+};\sigma)]\left(\mean{B_{\sigma}}H_{0}((u-h)_{+})\,dx\right)^{1+\frac{sp}{n}} \\
& \quad + \frac{c}{(k-h)^{p(p-1)}}r^{-tq}\left(\mean{B_{\sigma}}H_{0}((u-h)_{+})\,dx\right)^{p}.
\end{align*}

Now, for $i=0,1,2,...$ and $k_{0}>1$, we write 
\begin{equation*}
\sigma_{i} \coloneqq \frac{r}{2}(1+2^{-i}),\quad k_{i} \coloneqq 2k_{0}(1-2^{-i-1})\quad \textrm{and}\quad y_{i} \coloneqq \int_{A^{+}(k_{i},\sigma_{i})}H_{0}((u-k_{i})_{+})\,dx.
\end{equation*}
Since $H_{0}(u) \in L^{1}(\Omega)$ from the assumption \eqref{assumption.bdd}, we see that 
\[
y_0 = \int_{A^+(k_0,r)}H_{0}((u-k_0)_+) \,dx \ \ \longrightarrow\ \ 0 \quad \text{as }\ k_0\ \to \ \infty.
\]
First, we consider $k_0>1$ so large that
\[
y_{i}\le y_{i-1} \le \cdots \le y_0 \le  1, \quad i=1,2,\dots.
\]
Then, since
\begin{equation*}
T((u-k_{i})_{+};\sigma_{i}) \leq T(u;r/2) < \infty,
\end{equation*}
we have
\begin{align*}
y_{i+1} & \leq \tilde c\left(2^{iq^{2}/p}y_{i}^{q/p} + 2^{i(q/p' + 1)}y_{i}^{q/p} + 2^{i(sp^{2}/n+q)}y_{i}^{1+(sp/n)}  + 2^{i(sp^{2}/n+p)}y_{i}^{1+(sp/n)} + 2^{ip(p-1)}y_{i}^{p}\right) \\
& \leq \tilde c2^{\theta i}y_{i}^{1+\beta}
\end{align*}
for some constant $\tilde c >0$ depending on $\data$, $r$ and $T(u;r/2)$, where
\begin{equation*}
\theta = \max\left\{\frac{q^{2}}{p},\frac{q}{p'}+1,\frac{sp^{2}}{n}+q,p(p-1)\right\}, \qquad \beta = \min\left\{\frac{q}{p}-1,\frac{sp}{n},p-1\right\}.
\end{equation*}
Finally, we can choose 
$k_{0}$ so large that 
\begin{equation*}
y_{0} \leq \tilde{c}^{-1/\beta}2^{-\theta/\beta^{2}}
\end{equation*} 
holds. Then Lemma \ref{iterlem} implies
\begin{equation*}
y_{\infty} = \int_{A^{+}(2k_{0},r/2)}H_{0}((u-2k_{0})_{+})\,dx = 0,
\end{equation*}
which means that $u \leq 2k_{0}$ a.e. in $B_{r/2}$.

Applying the same argument to $-u$, we consequently obtain $u \in L^{\infty}(B_{r/2})$.
\end{proof}

\section{H\"older Continuity}\label{Holder}

Throughout this section, we assume that the modulating coefficient $a(\cdot,\cdot)$ satisfies \eqref{a.bound}-\eqref{a.holder}, and that a weak solution $u\in \mathcal{A}(\Omega)\cap L^{q-1}_{sp}(\R^n)$ under consideration is locally bounded in $\Omega$. We 
fix any $\Omega'\Subset \Omega$ and define
\begin{equation}\label{def.M}
M \equiv M(\Omega') \coloneqq 1+\|u\|_{L^\infty(\Omega')}^{q-p}.
\end{equation}

\subsection{Logarithmic estimate} We start with obtaining a logarithmic type estimate. This implies Corollary \ref{log.cor}, which will play a crucial role in the proof of H\"older continuity.
\begin{lemma}\label{log.lemma}
Under the assumptions in Theorem \ref{thm.hol}, 
let $u\in\mathcal{A}(\Omega)\cap L^{q-1}_{sp}(\mathbb{R}^{n})$ be a weak supersolution to \eqref{main.eq} such that $u \in L^{\infty}(\Omega')$ 
and $u \geq 0$ in a ball $B_{R} \equiv B_{R}(x_{0}) \subset \Omega'$ with $R<1$. 
Then the following estimate holds true for any $0 < \rho < R/2$ and $d>0$:
\begin{align*}
\int_{B_{\rho}}\int_{B_{\rho}}\left|\log\left(\frac{u(x)+d}{u(y)+d}\right)\right|\frac{dydx}{|x-y|^{n}}
& \leq c\widetilde M^{2}\Bigg(\rho^{n} + \rho^{n+sp}d^{1-p}\int_{\mathbb{R}^{n}\setminus B_{R}}\frac{u_{-}^{p-1}(y)+u_{-}^{q-1}(y)}{|y-x_{0}|^{n+sp}}\,dy \\
&\qquad\qquad  + \rho^{n+tq}d^{1-q}\int_{\mathbb{R}^{n}\setminus B_{R}(x_{0})}\frac{u_{-}^{q-1}(y)}{|y-x_{0}|^{n+tq}}\,dy \bigg)
\end{align*}
for some $c \equiv c(\data_1)$, where $\widetilde M \equiv \widetilde M  (\Omega') \coloneqq 1+(\lVert u \rVert_{L^{\infty}(\Omega')}+d)^{q-p}$.
\end{lemma}
\begin{proof}
We recall \eqref{def.H}, \eqref{def.h} and \eqref{phiell}, and further denote
\begin{equation*}\begin{aligned}
&\widetilde{H}(x,y,\tau) \coloneqq \frac{\tau^{p}}{\rho^{sp}} + a(x,y)\frac{\tau^{q}}{\rho^{tq}}, \quad \widetilde{h}(x,y,\tau) \coloneqq \frac{\tau^{p-1}}{\rho^{sp}} + a(x,y)\frac{\tau^{q-1}}{\rho^{tq}},\\
&G(\tau) \coloneqq \frac{\tau^{p}}{\rho^{sp}} + a_{2}\frac{\tau^{q}}{\rho^{tq}},
\quad 
 g(\tau) \coloneqq \frac{\tau^{p-1}}{\rho^{sp}} + a_{2}\frac{\tau^{q-1}}{\rho^{tq}},
\end{aligned}\end{equation*}
where $\tau\ge0$ and
\begin{equation*}
a_{2} \coloneqq \sup_{B_{2\rho}\times B_{2\rho}}a(\cdot,\cdot).
\end{equation*}
Let $\phi \in C^{\infty}_{0}(B_{3\rho/2})$ be a cut-off function satisfying $0 \leq \phi \leq 1$, $\phi \equiv 1$ in $B_{\rho}$ and $|D\phi| \le 4/\rho$.
Testing \eqref{weak.formulation} with $\varphi(x) = \phi^{q}(x)/g(u(x)+d)$, we have
\begin{align*}
0 & \leq \int_{B_{2\rho}}\int_{B_{2\rho}}\left[\Phi_{p}(u(x)-u(y))K_{sp}(x,y)\left(\frac{\phi^{q}(x)}{g(u(x)+d)}-\frac{\phi^{q}(y)}{g(u(y)+d)}\right)\right. \\
& \qquad\qquad \qquad +\left.a(x,y)\Phi_{q}(u(x)-u(y))K_{tq}(x,y)\left(\frac{\phi^{q}(x)}{g(u(x)+d)}-\frac{\phi^{q}(y)}{g(u(y)+d)}\right)\right]\,dxdy \\
& \qquad +2\int_{\R^n\setminus B_{2\rho}}\int_{B_{2\rho}}\left[\Phi_{p}(u(x)-u(y))K_{sp}(x,y)\frac{\phi^{q}(x)}{g(u(x)+d)}\right. \\
& \qquad \qquad\qquad \qquad \qquad +\left.a(x,y)\Phi_{q}(u(x)-u(y))K_{tq}(x,y)\frac{\phi^{q}(x)}{g(u(x)+d)}\right]\,dxdy \\
& \eqqcolon I_{1} + I_{2}.
\end{align*}
Moreover, in $I_1$ the integrand with respect to the measure $\frac{dxdy}{|x-y|^{n}}$ is denoted by $F(x,y)$, that is,
\begin{equation*}
I_1 = \int_{B_{2\rho}}\int_{B_{2\rho}} F(x,y)\, \frac{dxdy}{|x-y|^{n}},
\end{equation*}
\begin{align*}
F(x,y) & \coloneqq \Phi_{p}(u(x)-u(y))K_{sp}(x,y)|x-y|^{n}\left(\frac{\phi^{q}(x)}{g(u(x)+d)}-\frac{\phi^{q}(y)}{g(u(y)+d)}\right) \\
& \qquad + a(x,y)\Phi_{q}(u(x)-u(y))K_{tq}(x,y)|x-y|^{n}\left(\frac{\phi^{q}(x)}{g(u(x)+d)}-\frac{\phi^{q}(y)}{g(u(y)+d)}\right).
\end{align*}
We also denote $\bar{u}(x) \coloneqq u(x)+d$. Next, we estimate $I_1$ and $I_2$ separately. The remaining part of the proof is divided into four steps.

\textit{Step 1: Estimate of $F(x,y)$ when $\bar{u}(x) \geq \bar{u}(y) \geq \frac{1}{2}\bar{u}(x)$.}  We first observe that 
\begin{align*}
\frac{\phi^{q}(x)}{g(\bar{u}(x))}-\frac{\phi^{q}(y)}{g(\bar{u}(y))}  & = \frac{\phi^{q}(x)-\phi^{q}(y)}{g(\bar{u}(y))} + \phi^{q}(x)\left(\frac{1}{g(\bar{u}(x))}-\frac{1}{g(\bar{u}(y))}\right) \\
& \leq \frac{q\phi^{q-1}(x)|\phi(x)-\phi(y)|}{g(\bar{u}(y))} + \phi^{q}(x)\int_{0}^{1}\frac{d}{d\sigma}\left(\frac{1}{g(\sigma \bar{u}(x)+(1-\sigma)\bar{u}(y))}\right)\,d\sigma.
\end{align*}
To estimate the last integral, we first observe that
\begin{equation*}
\frac{d}{d\sigma}\left(\frac{1}{g(\sigma \bar{u}(x)+(1-\sigma)\bar{u}(y))}\right) = -\frac{g'(\sigma \bar{u}(x)+(1-\sigma)\bar{u}(y))}{g^{2}(\sigma\bar{u}(x)+(1-\sigma)\bar{u}(y))}(\bar{u}(x)-\bar{u}(y)),
\end{equation*}
where a direct calculation shows
\begin{equation*}
\frac{g'(\tau)}{g^{2}(\tau)} = \frac{(p-1)\dfrac{\tau^{p-2}}{\tau^{sp}}+(q-1)a_{2}\dfrac{\tau^{q-2}}{\rho^{tq}}}{\left(\dfrac{\tau^{p-1}}{\rho^{sp}}+a_{2}\dfrac{\tau^{q-1}}{\rho^{tq}}\right)^{2}}, 
\  \text{ hence }\ \frac{p-1}{G(\tau)} \leq \frac{g'(\tau)}{g^{2}(\tau)} \leq \frac{q-1}{G(\tau)}.
\end{equation*}
Thus we have
\begin{align*}
\frac{\phi^{q}(x)}{g(\bar{u}(x))}-\frac{\phi^{q}(y)}{g(\bar{u}(y))} & \leq \frac{q\phi^{q-1}(x)|\phi(x)-\phi(y)|}{g(\bar{u}(y))} - (p-1)\frac{\phi^{q}(x)(\bar{u}(x)-\bar{u}(y))}{G(\bar{u}(x))} \\
& \leq \frac{q\phi^{q-1}(x)|\phi(x)-\phi(y)|}{g(\bar{u}(y))} -\frac{p-1}{2^{q}}\frac{\phi^{q}(x)(\bar{u}(x)-\bar{u}(y))}{G(\bar{u}(y))}.
\end{align*}
Applying this inequality to $F(x,y)$ and using \eqref{kernel.growth}, we have  
\begin{equation}\label{bb}
\begin{aligned}
F(x,y) & \leq \Lambda q \frac{h(x,y,\bar{u}(x)-\bar{u}(y))\phi^{q-1}(x)|\phi(x)-\phi(y)|\bar{u}(y)}{G(\bar{u}(y))} \\
&\qquad - \frac{p-1}{2^{q}\Lambda}\frac{H(x,y,\bar{u}(x)-\bar{u}(y))\phi^{q}(x)}{G(\bar{u}(y))}.
\end{aligned}\end{equation}
Let us now estimate the first term in the right-hand side of \eqref{bb}. Applying Young's inequality to the numerator, for any small $\varepsilon>0$ we obtain
\begin{align*}
\lefteqn{ h(x,y,\bar{u}(x)-\bar{u}(y))\phi^{q-1}(x)|\phi(x)-\phi(y)|\bar{u}(y) } \\
& \leq \frac{\varepsilon(\bar{u}(x)-\bar{u}(y))^{p}\phi^{(q-1)p'}(x) + c(\varepsilon)|\phi(x)-\phi(y)|^{p}\bar{u}^{p}(y)}{|x-y|^{sp}} \\
& \quad +  a(x,y)\frac{\varepsilon(\bar{u}(x)-\bar{u}(y))^{q}\phi^{q}(x)+c(\varepsilon)|\phi(x)-\phi(y)|^{q}\bar{u}^{q}(y)}{|x-y|^{tq}} \\
& \leq \varepsilon\phi^{q}(x)H(x,y,\bar{u}(x)-\bar{u}(y)) \\
& \quad +  c(\varepsilon)\left(\frac{|\phi(x)-\phi(y)|^{p}\rho^{sp}}{|x-y|^{sp}}\frac{\bar{u}^{p}(y)}{\rho^{sp}} + a_{2}\frac{|\phi(x)-\phi(y)|^{q}\rho^{tq}}{|x-y|^{tq}}\frac{\bar{u}^{q}(y)}{\rho^{tq}}\right),
\end{align*}
where for the last inequality we have used that $x,y \in B_{2\rho}$. It then follows that
\begin{align*}
\lefteqn{ \frac{h(x,y,\bar{u}(x)-\bar{u}(y))\phi^{q-1}(x)|\phi(x)-\phi(y)|\bar{u}(y)}{G(\bar{u}(y))} }\\
& \leq \varepsilon\phi^{q}(x)\frac{H(x,y,\bar{u}(x)-\bar{u}(y))}{G(\bar{u}(y))} + c(\varepsilon)\left(\frac{\rho^{sp}}{|x-y|^{sp}}|\phi(x)-\phi(y)|^{p}+\frac{\rho^{tq}}{|x-y|^{tq}}|\phi(x)-\phi(y)|^{q}\right)
\end{align*}
for any small $\varepsilon > 0$.
Putting this into \eqref{bb} and choosing 
$$\varepsilon = \frac{p-1}{2^{q+1}q\Lambda},$$
we have
\begin{align*}
F(x,y) &\leq c\frac{|x-y|^{(1-s)p}}{\rho^{(1-s)p}}+c\frac{|x-y|^{(1-t)q}}{\rho^{(1-t)q}} -\frac{p-1}{2^{q+1}\Lambda}\frac{\phi^{q}(x)H(x,y,\bar{u}(x)-\bar{u}(y))}{G(\bar{u}(y))}.
\end{align*}
In order to estimate the last term in the above display, we note that 
\begin{equation*}
a_{2} = a_{2}-a(x,y) + a(x,y) \leq [a]_{\alpha}8^{\alpha}\rho^{\alpha} + a(x,y), \quad  x,y\in B_{2\rho},
\end{equation*} 
to discover
\begin{equation}\label{aa}
\begin{aligned}
G(\bar{u}(y)) & = \frac{\bar{u}^{p}(y)}{\rho^{sp}} + a_{2}\frac{\bar{u}^{q}(y)}{\rho^{tq}} \\
& \leq \frac{\bar{u}^{p}(y)}{\rho^{sp}} + [a]_{\alpha}8^{\alpha}\rho^{\alpha-tq+sp}\lVert \bar{u} \rVert_{L^{\infty}(\Omega')}^{q-p}\frac{\bar{u}^{p}(y)}{\rho^{sp}} + a(x,y)\frac{\bar{u}^{q}(y)}{\rho^{tq}} \\
& \leq (1+8^{\alpha}[a]_{\alpha})\left(1+(\lVert u \rVert_{L^{\infty}(\Omega')} + d)^{q-p} \right)\widetilde{H}(x,y,\bar{u}(y)),
\end{aligned}
\end{equation}
where we have used the inequality in \eqref{assumption.hol} with $\rho\le 1$. Then it follows that
\begin{equation*}
-\frac{\phi^{q}(x)H(x,y,\bar{u}(x)-\bar{u}(y))}{G(\bar{u}(y))} \leq -\frac{1}{c \widetilde M}\frac{\phi^{q}(x) H(x,y,\bar{u}(x)-\bar{u}(y))}{\widetilde{H}(x,y,\bar{u}(y))}
\end{equation*}
and therefore
\begin{equation}\label{midest.1st}
F(x,y) \leq c\frac{|x-y|^{(1-s)p}}{\rho^{(1-s)p}} + c\frac{|x-y|^{(1-t)q}}{\rho^{(1-t)q}} - \frac{1}{c \widetilde M }\frac{\phi^{q}(x) H(x,y,\bar{u}(x)-\bar{u}(y))}{\widetilde{H}(x,y,\bar{u}(y))}.
\end{equation}
We now need to 
derive an estimate for $\log \bar{u}$.
Observe 
\begin{align*}
\log \bar{u}(x) - \log \bar{u}(y)
 = \int_{0}^{1}\frac{\bar{u}(x)-\bar{u}(y)}{\sigma \bar{u}(x) + (1-\sigma)\bar{u}(y)}\,d\sigma  \leq \frac{\bar{u}(x)-\bar{u}(y)}{\bar{u}(y)}  = \frac{\dfrac{\bar{u}(x)-\bar{u}(y)}{|x-y|^{s}}}{\dfrac{\bar{u}(y)}{\rho^{s}}} \frac{|x-y|^{s}}{\rho^{s}}
\end{align*}
and use the monotonicity of the function
\begin{equation*}
\tau \mapsto \frac{\tau^{p}+a(x,y)\tau^{q}|x-y|^{-(t-s)q}}{\tau}, \quad \tau \ge 0,
\end{equation*}
to obtain
\begin{align*}
\lefteqn{ \log \bar{u}(x) - \log \bar{u}(y) }\\
& \leq \left[\frac{\left(\dfrac{\bar{u}(x)-\bar{u}(y)}{|x-y|^{s}}\right)^{p}+a(x,y)\left(\dfrac{\bar{u}(x)-\bar{u}(y)}{|x-y|^{s}}\right)^{q}\dfrac{1}{|x-y|^{(t-s)q}}}{\left(\dfrac{\bar{u}(y)}{\rho^{s}}\right)^{p}+a(x,y)\left(\dfrac{\bar{u}(y)}{\rho^{s}}\right)^{q}\dfrac{1}{|x-y|^{(t-s)q}}}+1\right]\dfrac{|x-y|^{s}}{\rho^{s}} \\
& \leq c\frac{H(x,y,\bar{u}(x)-\bar{u}(y))}{\widetilde{H}(x,y,\bar{u}(y))} + \frac{|x-y|^{s}}{\rho^{s}}.
\end{align*}
For the last inequality, we have used the fact that $|x-y| \leq 2\rho$. Finally, inserting this into \eqref{midest.1st}, we obtain
\begin{equation*}
F(x,y) \leq c\frac{|x-y|^{(1-s)p}}{\rho^{(1-s)p}} + c\frac{|x-y|^{(1-t)q}}{\rho^{(1-t)q}}+ c\frac{|x-y|^{s}}{\rho^{s}}  - \frac{\phi^{q}(x)}{c \widetilde M }\log \left(\frac{\bar u(x)}{\bar u(y)}\right).
\end{equation*}

\textit{Step 2: Estimate of $F(x,y)$ when $\bar{u}(x) \geq 2 \bar{u}(y)$.} We first observe from the second inequality in Lemma~\ref{size.comp} with $\varepsilon=\frac{2^{p-1}-1}{2^{p}}$ and $\bar{u}(x) \geq 2 \bar{u}(y)$ that 
\begin{align*}
\frac{\phi^{q}(x)}{g(\bar{u}(x))}-\frac{\phi^{q}(y)}{g(\bar{u}(y))} 
& = \frac{\phi^{q}(x)-\phi^{q}(y)}{g(\bar{u}(x))} + \phi^{q}(y)\left(\frac{1}{g(\bar{u}(x))} - \frac{1}{g(\bar{u}(y))}\right) \\
& \leq \frac{\phi^{q}(x)-\phi^{q}(y)}{g(\bar{u}(x))} + \phi^{q}(y)\left(\frac{1}{g(2\bar{u}(y))} - \frac{1}{g(\bar{u}(y))}\right) \\
& \leq \frac{\phi^{q}(x)-\phi^{q}(y)}{g(\bar{u}(x))} -\left(1-\frac{1}{2^{p-1}}\right) \frac{\phi^{q}(y)}{g(\bar{u}(y))} \\
& \leq \frac{\varepsilon\phi^{q}(y) + c(\varepsilon)|\phi(x)-\phi(y)|^{q}}{g(\bar{u}(x))} - \left(1-\frac{1}{2^{p-1}}\right) \frac{\phi^{q}(y)}{g(\bar{u}(y))} \\
& \leq c\frac{|\phi(x)-\phi(y)|^{q}}{g(\bar{u}(x))} - \frac{2^{p-1}-1}{2^{p}} \frac{\phi^{q}(y)}{g(\bar{u}(y))}.
\end{align*}
This implies
\begin{equation*}
F(x,y)  \leq c\frac{h(x,y,\bar{u}(x)-\bar{u}(y))|\phi(x)-\phi(y)|^{q}}{g(\bar{u}(x))} -\frac{1}{c}\frac{h(x,y,\bar{u}(x)-\bar{u}(y))\phi^{q}(y)}{g(\bar{u}(y))}
\end{equation*}
Estimating the right-hand side similarly as in \eqref{aa}, we find
\begin{align*}
F(x,y)  \leq c\frac{h(x,y,\bar{u}(x)-\bar{u}(y))|\phi(x)-\phi(y)|^{q}}{g(\bar{u}(x))}  - \frac{1}{c\widetilde M} \frac{h(x,y,\bar{u}(x)-\bar{u}(y))}{\widetilde{h}(x,y,\bar{u}(y))}.
\end{align*}
The first term in the right-hand side is estimated as
\begin{align*}
\lefteqn{ \frac{h(x,y,\bar{u}(x)-\bar{u}(y))|\phi(x)-\phi(y)|^{q}}{g(\bar{u}(x))} } \\
& = \frac{\dfrac{\rho^{sp}}{|x-y|^{sp}}\dfrac{|\bar{u}(x)-\bar{u}(y)|^{p-1}}{\rho^{sp}} + a(x,y)\dfrac{\rho^{tq}}{|x-y|^{tq}}\dfrac{|\bar{u}(x)-\bar{u}(y)|^{q-1}}{\rho^{tq}}}{\dfrac{|\bar{u}(x)-\bar{u}(y)|^{p-1}}{\rho^{sp}} + a_{2}\dfrac{|\bar{u}(x)-\bar{u}(y)|^{q-1}}{\rho^{tq}}}|\phi(x)-\phi(y)|^{q} \\
& \leq \left(\frac{\rho^{sp}}{|x-y|^{sp}} + \frac{\rho^{tq}}{|x-y|^{tq}}\right)\frac{|x-y|^{q}}{\rho^{q}},
\end{align*}
hence we have
\begin{equation}\label{midest.2nd}
F(x,y) \leq c\frac{|x-y|^{q-sp}}{\rho^{q-sp}} + c\frac{|x-y|^{(1-t)q}}{\rho^{(1-t)q}} - \frac{1}{c\widetilde M}\frac{h(x,y,\bar{u}(x)-\bar{u}(y))}{\widetilde{h}(x,y,\bar{u}(y))}.
\end{equation}
Furthermore, since we have in this case
\begin{align*}
\log \bar{u}(x) - \log \bar{u}(y) \leq \log \left(\frac{2(\bar{u}(x)-\bar{u}(y))}{\bar{u}(y)}\right) & \le c \left(\frac{2(\bar{u}(x)-\bar{u}(y))}{\bar{u}(y)}\right)^{p-1}\\
 &\le c \frac{\left(\dfrac{\bar{u}(x)-\bar{u}(y)}{|x-y|^{s}}\right)^{p-1}}{\left(\dfrac{\bar{u}(y)}{\rho^{s}}\right)^{p-1}}\frac{|x-y|^{s(p-1)}}{\rho^{s(p-1)}}.
\end{align*}
We then use the monotonicity of the function 
\begin{equation*}
\tau \mapsto \frac{\tau^{p-1} + a(x,y)\tau^{q-1}|x-y|^{-(t-s)q}}{\tau^{p-1}}, \quad \tau\ge 0, 
\end{equation*}
and use the fact that $|x-y| \leq 2\rho$, to have
\begin{align*}
\lefteqn{\log \bar{u}(x) - \log \bar{u}(y)}\\
 & \le c \left[\frac{\left(\dfrac{\bar{u}(x)-\bar{u}(y)}{|x-y|^{s}}\right)^{p-1} + a(x,y)\left(\dfrac{\bar{u}(x)-\bar{u}(y)}{|x-y|^{s}}\right)^{q-1}\dfrac{1}{|x-y|^{(t-s)q}}}{\left(\dfrac{\bar{u}(y)}{\rho^{s}}\right)^{p-1} + a(x,y)\left(\dfrac{\bar{u}(y)}{\rho^{s}}\right)^{q-1}\dfrac{1}{|x-y|^{(t-s)q}}} + 1 \right]\frac{|x-y|^{s(p-1)}}{\rho^{s(p-1)}} \\
& \le c \frac{h(x,y,\bar{u}(x)-\bar{u}(y))}{\widetilde{h}(x,y,\bar{u}(y))} + c\frac{|x-y|^{s(p-1)}}{\rho^{s(p-1)}}.
\end{align*}
Finally, inserting this into \eqref{midest.2nd}, we obtain
\begin{equation*}
F(x,y) \leq c\frac{|x-y|^{q-sp}}{\rho^{q-sp}} + c\frac{|x-y|^{(1-t)q}}{\rho^{(1-t)q}} + c \frac{|x-y|^{s(p-1)}}{\rho^{s(p-1)}} - \frac{1}{c\widetilde M} \log\left(\frac{\bar{u}(x)}{\bar{u}(y)}\right).
\end{equation*}

\textit{Step 3: Estimate of $I_1$.} From Step 1 and Step 2, we have that 
\begin{align*}
F(x,y) &\leq c\frac{|x-y|^{(1-s)p}}{\rho^{(1-s)p}} + c\frac{|x-y|^{(1-t)q}}{\rho^{(1-t)q}} + c\frac{|x-y|^{s}}{\rho^{s}} + c\frac{|x-y|^{s(p-1)}}{\rho^{s(p-1)}} \\
& \qquad - \frac{(\min\{\phi(x),\phi(y)\})^{q}}{c\widetilde M}\left|\log\left(\frac{\bar{u}(x)}{\bar{u}(y)}\right)\right|,
\end{align*}
when $\bar{u}(x) \ge \bar{u}(y)$. Moreover, by the symmetry of the above estimate for $x$ and $y$, the same estimate still holds when  $\bar{u}(x) < \bar{u}(y)$.  
Therefore, $I_{1}$ is finally estimated as follows:
\begin{equation}\label{est.I1}\begin{aligned}
I_{1} & \leq -\frac{1}{c \widetilde M}\int_{B_{\rho}}\int_{B_{\rho}}\left|\log\left(\frac{\bar{u}(x)}{\bar{u}(y)}\right)\right| \frac{dydx}{|x-y|^{n}} \\
& \quad + c\int_{B_{2\rho}}\int_{B_{4\rho}(x)}\left(\frac{|x-y|^{(1-s)p}}{\rho^{(1-s)p}}+\frac{|x-y|^{(1-t)q}}{\rho^{(1-t)q}}+\frac{|x-y|^{s}}{\rho^{s}}+\frac{|x-y|^{s(p-1)}}{\rho^{s(p-1)}}\right)\frac{dydx}{|x-y|^{n}} \\
& \leq -\frac{1}{c\widetilde M}\int_{B_{\rho}}\int_{B_{\rho}}\left|\log\left(\frac{\bar{u}(x)}{\bar{u}(y)}\right)\right| \frac{dydx}{|x-y|^{n}} + c\rho^{n}.
\end{aligned}\end{equation}

\textit{Step 4: Estimate of $I_{2}$ and Conclusion.} We start with the following observation:
\begin{itemize}
\item[(i)] If $y \in B_{R} \setminus B_{2\rho}$, then $u(y) \geq 0$ and $u(x)-u(y) \leq u(x)$;

\item[(ii)] If $y \in \mathbb{R}^{n} \setminus B_{R}$, then $(u(x)-u(y))_{+} \leq (u(x)+u_{-}(y))_{+} = u(x) + u_-(y)$.
\end{itemize}
Using this and the fact that $\supp\,\phi \subset B_{3\rho/2}$, we write
\begin{equation}\label{est.I2}
I_{2}  \leq 2\int_{B_{3\rho/2}}\int_{\mathbb{R}^{n}\setminus B_{2\rho}}\frac{h(x,y,u(x)+d)}{g(u(x)+d)}\,\frac{dydx}{|x-y|^{n}}  + 2\int_{B_{3\rho/2}}\int_{\mathbb{R}^{n}\setminus B_{R}}\frac{h(x,y,u_-(y))}{g(u(x)+d)}\,\frac{dydx}{|x-y|^{n}}.
\end{equation}
Since we are considering integrals over the complement of balls, we cannot directly compare $a_{2}$ and $a(x,y)$ there. In order to overcome this difficulty, we observe that \eqref{a.bound} and \eqref{assumption.hol} imply
\begin{equation} \label{a.control} 
\begin{aligned}
a(x,y)  \leq a(x,y) - a(x,x) + a_{2} & \leq \left(a(x,y)-a(x,x)\right)^{\frac{tq-sp}{\alpha}}\left(2\|a\|_{L^{\infty}}\right)^{1-\frac{tq-sp}{\alpha}} + a_{2} \\
& \le c |x-y|^{tq-sp} + a_{2},
\end{aligned}\end{equation}
whenever $x \in B_{2\rho}$ and $y \in \mathbb{R}^{n}$.

For the first integral in \eqref{est.I2}, 
we use \eqref{a.control} and the fact that  $|x-y| > \frac{\rho}{2}$ for $x \in B_{3\rho/2}$ and $y \in \mathbb{R}^{n}\setminus B_{2\rho}$ to find 
\begin{align*}
\frac{h(x,y,u(x)+d)}{g(u(x)+d)} \le c\frac{\dfrac{\bar{u}^{p-1}(x)}{|x-y|^{sp}}+a(x,y)\dfrac{\bar{u}^{q-1}(x)}{|x-y|^{tq}}}{\dfrac{\bar{u}^{p-1}(x)}{|x-y|^{sp}}+a_{2}\dfrac{\bar{u}^{q-1}(x)}{|x-y|^{tq}}}
 \le c\frac{\dfrac{\bar{u}^{p-1}(x) + \bar{u}^{q-1}(x)}{|x-y|^{sp}}+a_{2}\dfrac{\bar{u}^{q-1}(x)}{|x-y|^{tq}}}{\dfrac{\bar{u}^{p-1}(x)}{|x-y|^{sp}}+a_{2}\dfrac{\bar{u}^{q-1}(x)}{|x-y|^{tq}}} \le c \widetilde M,
\end{align*}
which gives
\begin{equation}\label{first}
\int_{B_{3\rho/2}}\int_{\mathbb{R}^{n}\setminus B_{2\rho}}\frac{h(x,y,u(x)+d)}{g(u(x)+d)}\,\frac{dydx}{|x-y|^{n}}  \le c \widetilde M \rho^{n} .
\end{equation}
For the second integral in \eqref{est.I2}, we use \eqref{a.control} and the fact that  $\frac{|y-x_{0}|}{|y-x|} \leq 1+\frac{|x-x_{0}|}{|y-x|} \leq 1 + \frac{3\rho/2}{\rho/2} = 4$ for $x\in B_{3\rho/2}$ and $y\in \R^n\setminus B_{2\rho}$ to find
\begin{equation*}
\begin{aligned}
\frac{h(x,y,u_{-}(y))}{g(u(x)+d)} \le \frac{\dfrac{u_{-}^{p-1}(y)}{|x-y|^{sp}}+a(x,y)\dfrac{u_{-}^{q-1}(y)}{|x-y|^{tq}}}{\dfrac{d^{p-1}}{\rho^{sp}}+a_{2}\dfrac{d^{q-1}}{\rho^{tq}}} 
& \le c\frac{\dfrac{u_{-}^{p-1}(y)}{|x-y|^{sp}}+|x-y|^{tq-sp}\dfrac{u_{-}^{q-1}(y)}{|x-y|^{tq}}+a_{2}\dfrac{u_{-}^{q-1}(y)}{|x-y|^{tq}}}{\dfrac{d^{p-1}}{\rho^{sp}}+a_{2}\dfrac{d^{q-1}}{\rho^{tq}}} \\
& \le c\rho^{sp}d^{1-p}\frac{u_{-}^{p-1}(y)+ u_{-}^{q-1}(y)}{|y-x_{0}|^{sp}} + c\rho^{tq}d^{1-q}\frac{u_{-}^{q-1}(y)}{|y-x_{0}|^{tq}}.
\end{aligned}\end{equation*}
Consequently, we obtain 
\begin{equation}\label{second}
\begin{aligned}
\lefteqn{ \int_{B_{3\rho/2}}\int_{\mathbb{R}^{n}\setminus B_{R}}\frac{h(x,y,u_{-}(y))}{g(u(x)+d)}\,\frac{dydx}{|x-y|^{n}} } \\
& \leq c\rho^{n+sp}d^{1-p}\int_{\mathbb{R}^{n}\setminus B_{R}}\frac{u_{-}^{p-1}(y)+u_{-}^{q-1}(y)}{|y-x_{0}|^{n+sp}}\,dy + c\rho^{n+tq}d^{1-q}\int_{\mathbb{R}^{n}\setminus B_{R}}\frac{u_{-}^{q-1}(y)}{|y-x_{0}|^{n+tq}}\,dy.
\end{aligned}
\end{equation}
Combining \eqref{est.I1}, \eqref{est.I2}, \eqref{first} and \eqref{second}, we finally get the desired estimate.
\end{proof}

The preceding lemma directly implies the following corollary.

\begin{corollary}\label{log.cor}
Under the same assumptions as in Lemma~\ref{log.lemma}, let $d,\zeta > 0$, $\xi>1$ and define
\begin{equation*}
v \coloneqq \min\{(\log(\zeta+d)-\log(u+d))_{+}, \log\xi\}.
\end{equation*}
Then we have
\begin{equation}\label{log.est}
\begin{aligned}
\mean{B_{\rho}}|v-(v)_{B_{\rho}}|\,dx
& \leq c \widetilde{M}^{2}  \left(1+ \rho^{sp}d^{1-p}\int_{\mathbb{R}^{n}\setminus B_{R}}\frac{u_{-}^{p-1}(y) + u_{-}^{q-1}(y)}{|y-x_{0}|^{n+sp}}\,dy\right.\\
&\hspace{4cm} \left. +\rho^{tq}d^{1-q}\int_{\mathbb{R}^{n}\setminus B_{R}}\frac{u_{-}^{q-1}(y)}{|y-x_{0}|^{n+tq}}\,dy\right) 
\end{aligned}
\end{equation}
for some $c\equiv c(\data_1)>0$, where $\widetilde{M} = 1+(\lVert u \rVert_{L^{\infty}(\Omega')}+d)^{q-p}$.
\end{corollary}
\begin{proof}
It suffices to observe that
\begin{align*}
\mean{B_{\rho}}|v-(v)_{B_{\rho}}|\,dx 
& \leq \mean{B_{\rho}}\mean{B_{\rho}}|v(x)-v(y)|\,dydx \\
& \leq c\rho^{-n}\int_{B_{\rho}}\int_{B_{\rho}}\frac{|\log(u(x)+d)-\log(u(y)+d)|}{|x-y|^{n}}\,dydx,
\end{align*}
as $v$ is a truncation of $\log(u+d)$. Now Lemma \ref{log.lemma} gives the desired result.
\end{proof}

\subsection{H\"older continuity: Proof of Theorem \ref{thm.hol}}
We are now in a position to prove Theorem \ref{thm.hol}. 
We first recall that $\Omega' \Subset \Omega$ has been fixed in the beginning of the section and the constant $M$ was defined in \eqref{def.M}.
We then fix a ball $B_{2r} \equiv B_{2r}(x_{0}) \subset \Omega'$. 
Let $\sigma \in (0,1/4]$ be a constant depending only on $\data_1$ and $\|u\|_{L^\infty(\Omega')}$ that satisfies 
\begin{equation}\label{sigma}
\sigma \le \min\left\{\frac{1}{4},2^{-\frac{2}{sp}}, 6^{-\frac{4(q-1)}{sq}}, \exp\left(-\frac{c_{*}M^{3}}{\nu_{*}}\right)\right\},
\end{equation}
where the large constant $c_{*} \equiv c_{*}(\data_{1})>0$ and the small one $\nu_{*} \equiv \nu_{*}(\data_{1},\|u\|_{L^{\infty}(\Omega')})>0$ are to be determined in \eqref{levelset.density} and \eqref{a0.small}, respectively, and then choose $\gamma\in(0,1)$ depending only on $\data_1$ and $\|u\|_{L^\infty(\Omega')}$ satisfying
\begin{equation}\label{gamma}
\gamma \le \min\left\{\log_{\sigma}\left(\frac{1}{2}\right),\frac{sp}{2(p-1)},\frac{tq}{2(q-1)},\log_{\sigma}\left(1-\sigma^{\frac{sq}{2(q-1)}}\right)\right\}.
\end{equation}
We define
\begin{equation}\label{K0}\begin{aligned}
\frac{1}{2}K_{0} & \coloneqq \sup_{B_{r}}|u| + \left[ r^{sp}\int_{\mathbb{R}^{n}\setminus B_{r}}\frac{|u(x)|^{p-1} + |u(x)|^{q-1}}{|x-x_{0}|^{n+sp}}\,dx \right]^{\frac{1}{p-1}} \\
& \hspace{5cm} + \left[ r^{tq}\int_{\mathbb{R}^{n}\setminus B_{r}}\frac{|u(x)|^{q-1}}{|x-x_{0}|^{n+tq}}\,dx \right]^{\frac{1}{q-1}}
\end{aligned}\end{equation}
and, for $j \in \mathbb{N}\cup\{0\}$, we write
\begin{equation*}
r_{j} \coloneqq \sigma^{j}r, \quad B_{j} \coloneqq B_{r_{j}}(x_{0})
\quad\text{and}\quad 
K_{j} \coloneqq \sigma^{\gamma j}K_{0}.
\end{equation*}

Now, 
we are going prove the following oscillation lemma, which implies $u\in C^{0,\gamma}(B_{r})$.
\begin{lemma}
Under the assumptions of Theorem \ref{thm.hol}, 
let $u$ be a weak solution to \eqref{main.eq}. Then we have for every $j\in \mathbb{N}\cup\{0\}$
\begin{equation}\label{osc.est}
\omega(r_{j}) \coloneqq \osc_{B_{j}}u \leq K_{j}.
\end{equation}
\end{lemma}
\begin{proof}
\textit{Step 1: Induction.} The proof goes by induction on $j$. For $j=0$ it is obvious from the definition of $K_{0}$.
Now we assume that \eqref{osc.est} holds for all $i \in \{0,...,j\}$ with some $j \geq 0$ and show that it holds also for $j+1$. That is, we will show that
\begin{equation}\label{j+1}
\omega(r_{j+1})\le K_{j+1}.
\end{equation}
Without loss of generality, we assume that 
\[
\omega(r_{j+1})\ge \frac{1}{2}K_{j+1}.
\]
Then, this together with the fact that $\sigma^\gamma\ge \frac{1}{2}$ from \eqref{gamma} implies that
\begin{equation}\label{Kj}
\omega(r_{j})\ge \omega(r_{j+1}) \ge \frac{1}{2}K_{j+1}= \frac{1}{2}\sigma^{\gamma} K_{j}\ge \frac{1}{4}K_j.
\end{equation}
  
We note that  either
\begin{equation}\label{u.large}
\frac{|2B_{j+1}\cap \{ u \geq \inf_{B_{j}}u + \omega(r_{j})/2 \}|}{|2B_{j+1}|} \geq \frac{1}{2}
\end{equation}
or
\begin{equation}\label{u.small}
\frac{|2B_{j+1}\cap \{ u \leq \inf_{B_{j}}u + \omega(r_{j})/2 \}|}{|2B_{j+1}|} \geq \frac{1}{2}
\end{equation}
must hold. We accordingly define
\begin{equation*}
u_{j} \coloneqq
\begin{cases}
u - \inf_{B_{j}}u & \textrm{if \eqref{u.large} holds}, \\
\sup_{B_{j}}u - u & \textrm{if \eqref{u.small} holds}.
\end{cases}
\end{equation*}
Then we have
\begin{equation}\label{density}
u_{j} \geq 0 \;\;\textrm{in}\;\; B_{j} \qquad \textrm{and} \qquad \frac{|2B_{j+1}\cap \{u_{j} \geq \omega(r_{j})/2\}|}{|2B_{j+1}|} \geq \frac{1}{2}.
\end{equation}
Moreover, $u_{j}$ is a weak solution to \eqref{main.eq} satisfying
\begin{equation}\label{uj.sup}
\sup_{B_{i}}|u_{j}| \leq \omega(r_{i}) \leq K_{i} \qquad \forall \; i\in \{0,...,j\}.
\end{equation}

\textit{Step 2: Tail estimates.}
We first claim that
\begin{equation}\label{p.tail}
r_{j}^{sp}\int_{\mathbb{R}^{n}\setminus B_{j}}\frac{|u_{j}(x)|^{p-1}+|u_{j}(x)|^{q-1}}{|x-x_{0}|^{n+sp}}\,dx  \leq cM \sigma^{\frac{sp}{2}}K_{j}^{p-1} 
\end{equation}
and
\begin{equation}\label{q.tail}
r_{j}^{tq}\int_{\mathbb{R}^{n}\setminus B_{j}}\frac{|u_{j}(x)|^{q-1}}{|x-x_{0}|^{n+tq}}\,dx \leq c \sigma^{\frac{tq}{2}}K_{j}^{q-1}
\end{equation}
for a constant $c \equiv c(\data_{1})$.
We will only give the proof of \eqref{p.tail}, since \eqref{q.tail} can be proved in almost the same way with $s$ and $p$ replaced by $t$ and $q$, respectively.  From  \eqref{uj.sup},  \eqref{K0} and \eqref{def.M}, we have
\begin{equation}\label{ahah} \begin{aligned}
\lefteqn{ r_{j}^{sp}\int_{\mathbb{R}^{n}\setminus B_{j}}\frac{|u_{j}(x)|^{p-1}+|u_{j}(x)|^{q-1}}{|x-x_{0}|^{n+sp}}\,dx } \\
& = r_{j}^{sp}\sum_{i=1}^{j}\int_{B_{i-1}\setminus B_{i}}\frac{|u_{j}(x)|^{p-1}+|u_{j}(x)|^{q-1}}{|x-x_{0}|^{n+sp}}\,dx + r_{j}^{sp}\int_{\mathbb{R}^{n}\setminus B_{0}} \frac{|u_{j}(x)|^{p-1}+|u_{j}(x)|^{q-1}}{|x-x_{0}|^{n+sp}}\,dx \\
& \leq \sum_{i=1}^{j}\left(\frac{r_{j}}{r_{i}}\right)^{sp}\left[\left(\sup_{B_{i-1}}|u_{j}|\right)^{p-1}+\left(\sup_{B_{i-1}}|u_{j}|\right)^{q-1}\right] + cM \left(\frac{r_{j}}{r_{1}}\right)^{sp}K_{0}^{p-1} \\
& \le cM\sum_{i=1}^{j}\left(\frac{r_{j}}{r_{i}}\right)^{sp}K_{i-1}^{p-1},
\end{aligned}\end{equation}
where for the first inequality we have used
\begin{align*}
\lefteqn{ r_{j}^{sp}\int_{\mathbb{R}^{n}\setminus B_{0}}\frac{|u_{j}(x)|^{p-1}+|u_{j}(x)|^{q-1}}{|x-x_{0}|^{n+sp}}\,dx }\\
& \le c \left(\frac{r_{j}}{r_{0}}\right)^{sp}\left[\left(\sup_{B_{0}}|u|\right)^{p-1} +  \left(\sup_{B_{0}}|u|\right)^{q-1}\right] +c r_{j}^{sp}\int_{\mathbb{R}^{n}\setminus B_{0}}\frac{|u(x)|^{p-1}+|u(x)|^{q-1}}{|x-x_{0}|^{n+sp}}\,dx \\
& \le cM \left(\frac{r_{j}}{r_{1}}\right)^{sp}K_{0}^{p-1}.
\end{align*}
Now the sum appearing in \eqref{ahah} is estimated as
\begin{align*}
\sum_{i=1}^{j}\left(\frac{r_{j}}{r_{i}}\right)^{sp}K_{i-1}^{p-1} 
& = K_{0}^{p-1}\left(\frac{r_{j}}{r_{0}}\right)^{\gamma(p-1)}\sum_{i=1}^{j}\left(\frac{r_{i-1}}{r_{i}}\right)^{\gamma(p-1)}\left(\frac{r_{j}}{r_{i}}\right)^{sp-\gamma(p-1)} \\
& = K_{j}^{p-1}\sigma^{-\gamma(p-1)}\sum_{i=1}^{j}\sigma^{i(sp-\gamma(p-1))} \\
& \le 2^{p-1} K_{j}^{p-1}\sum_{i=1}^{j}\sigma^{i\frac{sp}{2}}  \leq 2^{p-1} K_{j}^{p-1}\frac{\sigma^{\frac{sp}{2}}}{1-\sigma^{\frac{sp}{2}}}  
\leq 2^p\sigma^{\frac{sp}{2}}K_{j}^{p-1},
\end{align*}
where we have used the facts that $\sigma^{-\gamma}\le 2$, $sp-\gamma(p-1)\ge \frac{sp}{2}$ and $\sigma^{\frac{sp}{2}}\le \frac{1}{2}$ from \eqref{sigma} and \eqref{gamma}.

\textit{Step 3: A density estimate.}
We next apply Corollary \ref{log.cor} to the function
\begin{equation*}
v \coloneqq \min\left\{ \left[ \log\left(\frac{\omega(r_{j})/2 + d}{u_{j} + d}\right) \right]_{+}, k \right\},
\end{equation*}
where $k>0$ is to be chosen and 
\begin{equation}\label{depsilon}
d \equiv d_{j} \coloneqq \varepsilon K_{j} \qquad \textrm{with} \qquad  \varepsilon \coloneqq \sigma^{\frac{sq}{2(q-1)}}\ge \max\left\{\sigma^{\frac{sp}{2(p-1)}},\sigma^{\frac{tq}{2(q-1)}}\right\} .
\end{equation} 
Note that by \eqref{Kj} we see that
\begin{equation}\label{dM}
d_{j} \le 4\omega(r_j) \le 8\|u\|_{L^\infty(\Omega')}, \quad \text{hence} \quad \widetilde M \le c M. 
\end{equation}
Combining the resulting estimate \eqref{log.est} with \eqref{p.tail}-\eqref{q.tail}, we obtain
\begin{equation}\label{log.excess}
\begin{aligned}
\mean{2B_{j+1}}|v-(v)_{2B_{j+1}}|\,dx & \le c M^{2}\left[ 1+Md_{j}^{1-p}\sigma^{\frac{sp}{2}}K_{j}^{p-1} + d_{j}^{1-q}\sigma^{\frac{tq}{2}}K_{j}^{q-1} \right]\\
& \le c M^{3}\left[ 1+\left(d_{j}^{-1}\sigma^{\frac{sp}{2(p-1)}}K_{j}\right)^{p-1} + \left(d_{j}^{-1}\sigma^{\frac{tq}{2(q-1)}}K_{j}\right)^{q-1} \right]\\
&\leq cM^{3}
\end{aligned}
\end{equation}
for a constant $c\equiv c(\data_{1})>0$. In addition, 
we have from \eqref{density} that
\begin{align*}
k & = \frac{1}{|2B_{j+1}\cap\{u_{j} \geq \omega(r_{j})/2\}|}\int_{2B_{j+1}\cap\{v=0\}}(k-v)\,dx  \leq 2\mean{2B_{j+1}}(k-v)\,dx  = 2(k-(v)_{2B_{j+1}}).
\end{align*}
This inequality and \eqref{log.excess} imply
\begin{align*}
\frac{|2B_{j+1}\cap\{v=k\}|}{|2B_{j+1}|} & \leq \frac{2}{k|2B_{j+1}|}\int_{2B_{j+1}\cap\{v=k\}}(k-(v)_{2B_{j+1}})\,dx \\
& \leq \frac{2}{k}\mean{2B_{j+1}}|v-(v)_{2B_{j+1}}|\,dx  \leq \frac{cM^{3}}{k}.
\end{align*}
At this moment, we choose
\begin{equation*}
k = \log\left(\frac{\omega(r_{j})/2+\varepsilon \omega(r_{j})}{3\varepsilon \omega(r_{j})}\right) = \log\left(\frac{1/2+\varepsilon}{3\varepsilon}\right) \ge \log\left(\frac{1}{6\varepsilon}\right) \ge \log\left(\frac{1}{\sqrt{\varepsilon}}\right) = \frac{sq}{4(q-1)}\log\frac{1}{\sigma},
\end{equation*}
where we have used the fact that $\sqrt{\varepsilon}= \sigma^\frac{sq}{4(q-1)}\le \frac{1}{6}$ from \eqref{sigma},
to infer that
\begin{equation}\label{levelset.density}
\frac{|2B_{j+1}\cap\{ u_{j} \leq d_{j} \}|}{|2B_{j+1}|} \leq \frac{cM^{3}}{k} \leq \frac{c_{*}M^{3}}{\log(1/\sigma)}
\end{equation}
for a constant $c_{*}>0$ depending only on $\data_{1}$. 

\textit{Step 4: Iteration.}
Now we proceed with an iteration argument. For $i=0,1,2,...$, we set 
\begin{equation*}
\rho_{i} = (1+2^{-i})r_{j+1},\qquad \tilde{\rho}_{i} = \frac{\rho_{i}+\rho_{i+1}}{2},\qquad B^{i} = B_{\rho_{i}}, \qquad \tilde{B}^{i} = B_{\tilde{\rho}_{i}}
\end{equation*}
and choose corresponding cut-off functions satisfying
\begin{equation*}
\phi_{i} \in C^{\infty}_{0}(\tilde{B}^{i}),\qquad 0\leq\phi\leq1, \qquad \phi_{i} \equiv 1 \textrm{ on } B^{i+1}, \qquad \textrm{and} \qquad |D\phi_{i}| \leq 2^{i+2} r_{j+1}^{-1}.
\end{equation*}
Furthermore, we set
\begin{equation*}
k_{i} = (1+2^{-i})d_{j}, \qquad w_{j} = (k_{i}-u_{j})_{+}
\end{equation*}
and
\begin{equation*}
A_{i} = \frac{|B^{i} \cap \{ u_{j} \leq k_{i} \}|}{|B^{i}|} = \frac{|B^{i} \cap \{ w_{i} \geq 0 \}|}{|B^{i}|}.
\end{equation*}
Notice that
\begin{equation}\label{radii.size}
r_{j+1} < \rho_{i+1} \leq \rho_{i} \leq 2r_{j+1}, \quad d_{j} \leq k_{i+1} \leq k_{i} \leq 2d_{j} \quad \textrm{and} \quad 0 \leq w_{i} \leq k_{i} \leq 2d_{j}.
\end{equation}

We then denote
\[
a_{1} \coloneqq \inf_{B_{2r_{j+1}}\times B_{2r_{j+1}}}a(\cdot,\cdot),
\qquad
a_{2} \coloneqq \sup_{B_{2r_{j+1}}\times B_{2r_{j+1}}}a(\cdot,\cdot)
\] 
and 
\begin{equation*}
G(\tau) \coloneqq \frac{\tau^{p}}{r_{j+1}^{sp}} + a_{2}\frac{\tau^{q}}{r_{j+1}^{tq}}.
\end{equation*} 
Using the first inequality in \eqref{radii.size} 
and applying Lemma~\ref{ineq2} with $f \equiv w_{i}$,
we obtain 
\begin{equation}\label{lhs.ith}
\begin{aligned}
A_{i+1}^{1/\kappa}G(k_{i}-k_{i+1}) & = \left[\frac{1}{|B^{i+1}|}\int_{B^{i+1}\cap \{ u_{j} \leq k_{i+1} \}}[G(k_{i}-k_{i+1})]^{\kappa}\,dx\right]^{\frac{1}{\kappa}}  \\
& \leq \left[ \mean{B^{i+1}} [G(w_{i})]^{\kappa}\,dx \right]^{\frac{1}{\kappa}} \\
& \le cM\mean{B^{i+1}}\int_{B^{i+1}}H(x,y,|w_{i}(x)-w_{i}(y)|)\frac{dxdy}{|x-y|^{n}} + cMG(d_{j})A_{i},
\end{aligned}
\end{equation}
where for the last inequality we have also used the following  estimate: 
\begin{equation*}
\mean{B^{i}}\left|\frac{w_{i}}{\rho_{i+1}^{s}}\right|^{p}+a_{1}\left|\frac{w_{i}}{\rho_{i+1}^{t}}\right|^{q}\,dx 
\le c  \left[\left(\frac{d_{j}}{r_{j+1}^{s}}\right)^{p} + a_{2}\left(\frac{d_{j}}{r_{j+1}^{t}}\right)^{q}\right]\frac{|B^{i}\cap\{u_{j} \leq k_{i}\}|}{|B^{i}|} = cG(d_{j})A_{i},
\end{equation*}
which is immediate from the definitions of $w_i$, $\rho_i$ and $A_i$. In order to estimate the integral on the right-hand side, we apply Lemma \ref{caccioppoli.est} to $w_{i}$ and $\phi_{i}$ in the ball $B^{i}$ (see also Remark \ref{rmk.caccioppoli}). 
Moreover, we estimate the tail term in the right-hand side as in the proof of Lemma \ref{log.lemma} by using \eqref{a.control}:
\begin{equation}\label{caccioppoli.ith}
\begin{aligned}
\lefteqn{\mean{B^{i+1}}\int_{B^{i+1}}H(x,y,|w_{i}(x)-w_{i}(y)|)\,\frac{dxdy}{|x-y|^{n}}} \\
& \le c \mean{B^{i}}\int_{B^{i}}H(x,y,(w_{i}(x)+w_{i}(y))|\phi_{i}(x)-\phi_{i}(y)|)\,\frac{dxdy}{|x-y|^{n}} \\
& \quad + c\left(\sup_{y \in \tilde{B}^{i}}\int_{\mathbb{R}^{n}\setminus B^{i}}\frac{w_{i}^{p-1}(x)+w_{i}^{q-1}(x)}{|x-y|^{n+sp}}+a_{2}\frac{w_{i}^{q-1}(x)}{|x-y|^{n+tq}}\,dx\right)\mean{B^{i}}w_{i}(x)\phi_{i}^{q}(x)\,dx.
\end{aligned}
\end{equation}
We estimate the terms in the right-hand side of \eqref{caccioppoli.ith} separately. By the definitions of $w_i$ and $\phi_i$, we have 
\begin{equation}\label{wi.ith}
\begin{aligned}
&\mean{B^{i}}\int_{B^{i}}H(x,y,(w_{i}(x)+w_{i}(y))|\phi_{i}(x)-\phi_{i}(y)|)\frac{dxdy}{|x-y|^{n}}\\
& \le c 2^{ip}r_{j+1}^{-p}k_{i}^{p}\frac{1}{|B^{i}|}\int_{B^{i}\cap\{u_{j}\leq k_{i}\}}\int_{B^{i}}\frac{1}{|x-y|^{n+(s-1)p}}\,dydx \\
& \quad + c 2^{iq}a_{2}r_{j+1}^{-q}k_{i}^{q}\frac{1}{|B^{i}|}\int_{B^{i}\cap\{u_{j}\leq k_{i}\}}\int_{B^{i}}\frac{1}{|x-y|^{n+(t-1)q}}\,dydx \\
& \le c2^{iq}\frac{|B^{i}\cap\{u_{j} \leq k_{i}\}|}{|B^{i}|}\left( r_{j+1}^{-sp}d_{j}^{p}+ a_{2}r_{j+1}^{-tq}d_{j}^{q}\right)\\
& =c 2^{iq}G(d_{j})A_{i},
\end{aligned}
\end{equation}
and 
\begin{equation}\label{rhs.ith}
\mean{B^{i}}w_{i}(x)\phi_{i}^{q}(x)\,dx \le c d_{j}A_{i}.
\end{equation}
As for the tail term, we first observe the following facts:
$\frac{|x-x_{0}|}{|x-y|} \leq 1+\frac{|y-x_{0}|}{|x-y|} \leq 1 + \frac{2r_{j+1}}{2^{-(i+1)r_{j+1}}} = 2^{i+3}$ 
for $x\in \R^n\setminus B^i$ and $y \in \tilde{B}^{i}$; $w_i\le k_i \le 2d_j$ in $B_j$; and $w_i \le k_i + |u_j| \le 2d_j+|u_j|$ in $\R^n\setminus B_j$.
Using these facts, \eqref{p.tail}, \eqref{q.tail}, \eqref{depsilon} and \eqref{dM}, we see that
\begin{equation}\label{Tail.ith}
\begin{aligned}
\lefteqn{ \sup_{y\in\tilde{B}^{i}}\int_{\mathbb{R}^{n}\setminus B^{i}}\frac{w_{i}^{p-1}(x) + w_{i}^{q-1}(x)}{|x-y|^{n+sp}}+a_{2}\frac{w_{i}^{q-1}(x)}{|x-y|^{n+tq}}\,dx } \\
& \le c 2^{i(n+tq)}\int_{\mathbb{R}^{n}\setminus B_{j+1}}\frac{w_{i}^{p-1}(x)+w_{i}^{q-1}(x)}{|x-x_{0}|^{n+sp}} +a_2\frac{w_{i}^{q-1}(x)}{|x-x_{0}|^{n+tq}}\,dx \\
& \le c 2^{i(n+tq)}\int_{\R^n\setminus B_{j+1}}\frac{d_{j}^{p-1}+d_{j}^{q-1}}{|x-x_{0}|^{n+sp}} +a_2\frac{d_{j}^{q-1}}{|x-x_{0}|^{n+tq}}\,dx \\
&\quad + c 2^{i(n+tq)}\int_{\mathbb{R}^{n}\setminus B_{j}}\frac{|u_{j}(x)|^{p-1}+|u_{j}(x)|^{q-1}}{|x-x_{0}|^{n+sp}} +a_2\frac{|u_{j}(x)|^{q-1}}{|x-x_{0}|^{n+tq}}\,dx \\
& \le c 2^{i(n+tq)}M\left(\frac{d_{j}^{p-1}}{r_{j+1}^{sp}}+a_{2}\frac{d_{j}^{q-1}}{r_{j+1}^{tq}}\right)  +c 2^{i(n+tq)}M\left(\frac{\sigma^{\frac{sp}{2}}}{\varepsilon^{p-1}}\frac{d_{j}^{p-1}}{r_{j+1}^{sp}} + a_{2}\frac{\sigma^{\frac{tq}{2}}}{\varepsilon^{q-1}}\frac{d_{j}^{q-1}}{r_{j+1}^{tq}}\right) \\
& \le c 2^{i(n+tq)}M\frac{G(d_{j})}{d_{j}}.
\end{aligned}\end{equation}

Therefore, combining  \eqref{lhs.ith}, \eqref{caccioppoli.ith}, \eqref{wi.ith}, \eqref{rhs.ith} and \eqref{Tail.ith}, we arrive at
\begin{align*}
A_{i+1}^{1/\kappa} G(2^{-i-1}d_{j}) = A_{i+1}^{1/\kappa} G(k_{i}-k_{i+1})  
\leq c2^{i(n+tq+q)}M^{2}G(d_{j})A_{i},
\end{align*}
which implies
\begin{equation*}
A_{i+1}\leq c_{0}2^{i\kappa(n+tq+2q)}M^{2\kappa}A_{i}^{\kappa}
\end{equation*}
for a constant $c_{0}>0$ depending only on $\data_{1}$.
In order to apply Lemma \ref{iterlem}, it should be guaranteed that
\begin{equation}\label{a0.small}
A_{0} \leq (c_{0}M^{2\kappa})^{-1/(\kappa-1)}2^{-(n+tq+2q)\kappa/(\kappa-1)^{2}} \eqqcolon \nu_{*}.
\end{equation}
This inequality holds by \eqref{levelset.density} and \eqref{sigma}. More precisely, we have 
\begin{equation*}
A_{0} = \frac{|2B_{j+1}\cap\{ u_{j} \leq 2d_{j} \}|}{|2B_{j+1}|} \leq \frac{c_{*}M^{3}}{\log(1/\sigma)} \le \nu_*.
\end{equation*}
Hence it follows that
$A_{i}\rightarrow 0$ as $i\rightarrow\infty$, which means that
\begin{equation*}
u_{j} \geq d_{j} =\varepsilon K_{j} \qquad \textrm{a.e. in}\;\; B_{j+1}.
\end{equation*}
From this with \eqref{uj.sup} and \eqref{gamma}, we finally obtain \eqref{j+1} as follows:
\begin{equation*}
\omega(r_{j+1}) = \sup_{B_{j+1}}u_{j} - \inf_{B_{j+1}}u_{j} \leq (1-\varepsilon)K_{j} = \left(1-\sigma^{\frac{tq}{2(q-1)}}\right)\sigma^{-\gamma}K_{j+1} \le K_{j+1}.
\end{equation*}
\end{proof}

\providecommand{\bysame}{\leavevmode\hbox to3em{\hrulefill}\thinspace}
\providecommand{\MR}{\relax\ifhmode\unskip\space\fi MR }
\providecommand{\MRhref}[2]{%
  \href{http://www.ams.org/mathscinet-getitem?mr=#1}{#2}
}
\providecommand{\href}[2]{#2}

\end{document}